\newcommand{\si}[1]{#1}
\newcommand{\jo}[1]{}

\si{
\documentclass[10pt,a4paper]{article}
\usepackage[utf8]{inputenc}
\usepackage{graphicx}
\usepackage{amsfonts}
\usepackage{amsthm}
\usepackage{amsmath}
\usepackage{amssymb}
\usepackage{array}
\usepackage{xcolor}
\usepackage{algorithm}
\usepackage{indentfirst}
\usepackage{algorithmic}
\usepackage{hyperref}
\usepackage{chngcntr}
\usepackage[left=2cm,right=2cm,top=2cm,bottom=2cm]{geometry}

\tracingstats=0

\newtheorem{theorem}{Theorem}[section]
\newtheorem{proposition}{Proposition}[section]

\newtheorem{lemma}{Lemma}[section]
\newtheorem{definition}{Definition}[section]
\newtheorem{example}{Example}[section]
\newtheorem{remark}{Remark}[section]

\hypersetup{
  colorlinks=true,
  citecolor=blue,
  linkcolor=blue,
  filecolor=magenta,      
  urlcolor=cyan,
}
}

\jo{
\RequirePackage{fix-cm}
\documentclass[smallextended]{svjour3}       % onecolumn (second format)
\smartqed  % flush right qed marks, e.g. at end of proof

\usepackage[utf8]{inputenc}
\usepackage{graphicx}
\usepackage{amsfonts}
\usepackage{amsmath}
\usepackage{amssymb}
\usepackage{array}
\usepackage{xcolor}
\usepackage{indentfirst}
\usepackage{algorithm}
\usepackage{algorithmic}
\usepackage{chngcntr}
\usepackage{natbib,hyperref}

\tracingstats=0
}

\usepackage{tikz}
\usepackage[font=small,labelfont=bf]{caption}

\renewcommand{\S}{\mathbb{S}}					 % linear space S

\newcommand{\seq}[1]{\{#1^k\}_{k\in \N}}		 % sequence
\renewcommand{\bar}{\overline}                   % bigger bar (just renaming it)
\newcommand{\I}{\mathbb{I}}                      % identity matrix
\newcommand{\lin}{\mathrm{lin}}                  % lineality space
\newcommand{\Ker}{\mathrm{Ker}\hspace{0.05cm}}   % kernel
\renewcommand{\Im}{\mathrm{Im}\hspace{0.05cm}}   % image
\newcommand{\tr}{\mathrm{trace}}                 % trace
\newcommand{\Diag}{\mathrm{Diag}}                % Diag operator
\renewcommand{\int}{\mathrm{int}\hspace{0.05cm}} % interior
\newcommand{\rank}{\mathrm{rank}}                % rank
\newcommand{\F}{\mathcal{F}}                     % feasible set                       
\newcommand{\norm}[1]{\|#1\|}                    % norm
\newcommand{\enorm}[1]{\|#1\|_2}                 % 2-norm
\newcommand{\projs}{\Pi_{\mathbb{S}^m_+}}        % projection onto m-dim SDC
\newcommand{\R}{\mathbb{R}}  					 % real space (version 1)
\newcommand{\T}{\top\hspace{-1pt}}               % transpose
\newcommand{\N}{\mathbb{N}}                      % naturals
\newcommand{\xb}{\bar{x}}                        % x-bar                       
\newcommand{\Yb}{\bar{Y}}                        % Y-bar  
\newcommand{\sym}{\mathbb{S}^m}                  % semidefinite cone

\si{
\usepackage[normalem]{ulem}								% s.t. text
}

\jo{
\usepackage[misc]{ifsym}
\newcommand{\corr}{${\textnormal{\Letter}}$}
}

\begin{document}

%███████╗████████╗ █████╗ ██████╗ ████████╗
%██╔════╝╚══██╔══╝██╔══██╗██╔══██╗╚══██╔══╝
%███████╗   ██║   ███████║██████╔╝   ██║   
%╚════██║   ██║   ██╔══██║██╔══██╗   ██║   
%███████║   ██║   ██║  ██║██║  ██║   ██║   
%╚══════╝   ╚═╝   ╚═╝  ╚═╝╚═╝  ╚═╝   ╚═╝   

\title{Weak notions of nondegeneracy in nonlinear semidefinite programming
 \footnotetext{The authors received financial support from
 CEPID - FAPESP (grant 
 2013/07375-0),  
FAPESP (grants 
2018/24293-0, 
2017/18308-2, and 
2017/17840-2), 
CNPq (grants 
301888/2017-5, 
303427/2018-3, and 
404656/2018-8), 
PRONEX - CNPq/FAPERJ (grant 
E-26/010.001247/2016), and 
FONDECYT grant 
1201982 and 
Centro de Modelamiento Matem{\'a}tico (CMM), ACE210010 and FB210005, BASAL funds for center of excellence, all from ANID-Chile.}}

\jo{\titlerunning{Weak notions of nondegeneracy in nonlinear SDP}}

\si{
\author{
Roberto Andreani \thanks{Department of Applied Mathematics, State University of Campinas, Campinas, SP, Brazil. 
	Email: {\tt andreani@unicamp.br}}
\and
Gabriel Haeser \thanks{Department of Applied Mathematics, University of S{\~a}o Paulo, S{\~a}o Paulo, SP, Brazil. 
	Emails: {\tt ghaeser@ime.usp.br, leokoto@ime.usp.br}}	
\and 
Leonardo M. Mito \footnotemark[2]
\and
H{\'e}ctor Ram{\'i}rez \thanks{Departamento de Ingenier{\'i}a Matem{\'a}tica and Centro de Modelamiento Matem{\'a}tico (AFB170001 - CNRS UMI 2807), Universidad de Chile, Santiago, Chile.
Email: {\tt hramirez@dim.uchile.cl}}
}
}

\jo{
\author{
Roberto Andreani
\and
Gabriel Haeser
\and 
Leonardo M. Mito
\and
H{\'e}ctor Ram{\'i}rez
}

\authorrunning{R. Andreani, G. Haeser, L. M. Mito, and H. Ram{\'i}rez} % if too long for running head

\institute{ 
    Roberto Andreani \at
    Department of Applied Mathematics, 
    State University of Campinas, 
    S{\~ao} Paulo, Brazil. \\
	\email{andreani@unicamp.br}
    \and
    Gabriel Haeser \at
    Department of Applied Mathematics, 
    University of S\~ao Paulo, 
    S{\~a}o Paulo, Brazil.\\
    \email{ghaeser@ime.usp.br} 
    \and
    Leonardo M. Mito \corr 
    \at
    Department of Applied Mathematics, 
    University of S{\~a}o Paulo, 
    S{\~a}o Paulo, Brazil.\\
	\email{leokoto@ime.usp.br}
	\and	
	H{\'e}ctor Ram{\'i}rez \at
	Department of Mathematical Engineering and Center for Mathematical Modeling (CNRS UMI 2807),
	University of Chile, 
	Santiago, Chile.\\
	\email{hramirez@dim.uchile.cl}
  }
}

\jo{\date{Received: date / Accepted: date}}
%The correct dates will be entered by the editor.

\maketitle

\abstract{The constraint nondegeneracy condition is one of the most relevant and useful constraint qualifications in nonlinear semidefinite programming. It can be characterized in terms of any fixed orthonormal basis of the, let us say, $\ell$-dimensional kernel of the constraint matrix, by the linear independence of a set of $\ell(\ell+1)/2$ derivative vectors. We show that this linear independence requirement can be equivalently formulated in a smaller set, of $\ell$ derivative vectors, by considering all orthonormal bases of the kernel instead. This allows us to identify that not all bases are relevant for a constraint qualification to be defined, giving rise to a strictly weaker variant of nondegeneracy related to the global convergence of an external penalty method. We use some of these ideas to revisit an approach of Forsgren [Math. Prog. 88, 105--128, 2000] for exploiting the sparsity structure of a transformation of the constraints to define a constraint qualification, which led us to develop another relaxed notion of nondegeneracy using a simpler transformation. If the zeros of the derivatives of the constraint function at a given point are considered, instead of the zeros of the function itself in a neighborhood of that point, we obtain an even weaker constraint qualification that connects Forsgren's condition and ours. 
%In particular, both our new constraint qualifications reduce to the linear independence constraint qualification for nonlinear programming when considering a diagonal semidefinite constraint. More generally, when the problem has a diagonal block structure, the conditions formulated as a single block diagonal matrix constraint are equivalent to their analogues formulated with several semidefinite matrices.

\si{
\

}
}

\jo{

%\PACS{PACS code1 \and PACS code2 \and more}
}

%██╗███╗   ██╗████████╗██████╗  ██████╗ 
%██║████╗  ██║╚══██╔══╝██╔══██╗██╔═══██╗
%██║██╔██╗ ██║   ██║   ██████╔╝██║   ██║
%██║██║╚██╗██║   ██║   ██╔══██╗██║   ██║
%██║██║ ╚████║   ██║   ██║  ██║╚██████╔╝
%╚═╝╚═╝  ╚═══╝   ╚═╝   ╚═╝  ╚═╝ ╚═════╝ 

\section{Introduction}\label{sec:intro}

The study of \textit{linear} and \textit{nonlinear} \textit{semidefinite programming} (for short, SDP and NSDP, respectively) problems has been consistently growing over the last decades. There are several models for real world problems that can be reformulated as SDPs or NSDPs (we refer to the handbooks~\cite[Part 4]{handbookconic} and~\cite[Part 3]{handbookSDP} for a vast collection of applications), which motivate and are motivated by the development of theoretical results regarding optimality conditions and \textit{constraint qualifications} (CQs) for (N)SDPs. Loosely speaking, CQs are assumptions over the feasible set of an optimization problem that ensure that it can be locally described in terms of its first-order approximation. This leads to the possibility of characterizing all solutions of an (N)SDP problem in terms of the derivatives of the functions that describe it, which gives CQs a pivotal role in building convergence theories for practical algorithms. The standard way to do this is to prove that every feasible limit point of the output sequence of the algorithm satisfies the Karush-Kuhn-Tucker (KKT) conditions under a given CQ. Thus, employing a weaker CQ leads to a more robust convergence theory.

One of the most relevant CQs in the literature of (N)SDP is the so-called \textit{nondegeneracy}  (or \textit{transversality)}  \textit{condition}, introduced by Shapiro and Fan in \cite[Sec. 2]{shapfan} in the context of eigenvalue optimization, and later reformulated by Shapiro~\cite[Def. 4]{Shapiro1997} for general NSDPs. This condition has been widely used for characterizing sensitivity results (see, for instance, \cite{Fusek,KK13,MordNghiaRock14,MOR15,Mordukhovich2015a,Sun2006}), and also for proving global convergence and the rate of convergence of numerical algorithms (we refer to Yamashita and Yabe~\cite[Secs. 3, 4, and 5]{yamashitasdpreview} for a survey on this topic). However, it is known that even in the linear case, the solutions of large scale SDP problems tend to be degenerate, even though nondegeneracy is expected to hold in a generic sense. Besides, when the constraint of an NSDP problem has some sparsity structure near one of its solutions -- for instance, a diagonal structure -- then nondegeneracy is not satisfied at that solution~\cite{Shapiro1997}. This means that the convergence theory of an algorithm supported by nondegeneracy does not cover such points. 

The explanation for such kind of issue, in our opinion, is the degree of generality of the nondegeneracy condition. That is, although it was born in NSDP, nondegeneracy does not capture any particularity of the constraints, being straigthforwardly extended for any general conic optimization problem, as long as the cone is closed and convex. However, embedding specific traits of matrix-valued functions into nondegeneracy may be more or less direct, depending on how it is characterized. For example, it is well known that (block-)diagonal problems can be remodelled as multiple potentially dense constraints, such that the nondegeneracy condition, when applied to this remodelled problem, may hold. But what about other types of sparsity? While this question has once been addressed by Forsgren~\cite{Forsgren2000}, his approach is somewhat intricate and it was not the main topic of his paper, leaving room for a more dedicated analysis. In this paper, instead of defining nondegeneracy as the transversality of two particular subspaces -- which is the most usual definition -- we exploit an equivalent characterization by Shapiro~\cite[Prop. 6]{Shapiro1997}, which is phrased in terms of the gradients of the entries of an isolated ``active block'' of the constraints. One particularly interesting detail about this characterization is that it treats all representations of such an ``active block'' equally, but we show that some of them are more meaningful than others.

The contributions of this paper revolve around the following results:
\begin{itemize}
	\item We provide a new characterization of nondegeneracy that induces a weaker variant of it, here called \textit{weak-nondegeneracy}, which uses information of the eigenvectors of the constraints evaluated at nearby points;
   \item We incorporate a sparsity treatment in \cite[Prop. 6]{Shapiro1997}, which leads to another weak variant of nondegeneracy, called \textit{sparse-nondegeneracy}.
   \item We connect sparse-nondegeneracy with Forsgren's CQ by means of replacing, in both conditions, the strucutural zeros of the constraint function in a neighborhood of a point, with the zeros of the gradients of its entries at such point. This new condition happens to be a constraint qualification also, which we call \textit{gradient sparse-nondegeneracy}.
\end{itemize}
   These conditions are designed with the sole goal of assisting in proving global convergence of algorithms by means of sequential optimality conditions \cite{ahm10,ahv}; however, we envision that they may be further employed in sensitivity analysis, second-order analysis, among other applications.
   All variants of nondegeneracy we present are proved to be constraint qualifications strictly weaker than nondegeneracy. We also show that when our conditions are applied to diagonal matrices, they are reduced to the \textit{linear independence constraint qualification} (LICQ) from nonlinear programming (NLP). More generally, the conditions are invariant to block representations of (N)SDP problems as a single semidefinite block diagonal matrix or as multiple semidefinite constraints. Then, we compare our definitions with other CQs from the literature.
   
  	This paper is structured as follows: In Section~\ref{sec:common}, we introduce our notation; in Section~\ref{sec:ndg} we recall the nondegeneracy condition and we prove a new characterization of it, which is where the definition of weak-nondegeneracy comes from. In Section~\ref{sec:sndg}, we present our definition of sparse-nondegeneracy and a relaxation of it with distinct properties.  Finally, in Section~\ref{sec:conc}, we discuss some possibilities for prospective work.

\section{Preliminaries}\label{sec:common}

Let $f\colon\R^{n}\to\R$ and $G \colon\R^n \to \sym $ be continuously differentiable functions, where $\sym $ is the linear space of all $m\times m$ symmetric matrices, and let $\sym_+$ be the closed convex pointed cone of all $m\times m$ positive semidefinite matrices. The problem of interest in this paper is the following:
%
%Single-NSDP
\begin{equation}
  \tag{NSDP}
  \begin{aligned}
    & \underset{x \in \mathbb{R}^{n}}{\text{Minimize}}
    & & f(x), \\
    & \text{subject to}
    %& & h(x)=0,\\
    & & G (x) \succeq 0,
  \end{aligned}
  \label{NSDP}
\end{equation}
where $\succeq$ is the partial order induced by $\sym_+$, characterized by the relation: $M\succeq N$ $\Leftrightarrow$ $M-N\in \sym_+$, for all $M,N\in \sym$. It is worth pointing out that all results in this paper can be straightforwardly extended to NSDP problems with separate equality constraints, but we omit them for simplicity. The feasible set of \eqref{NSDP} will be denoted by $\F\doteq G^{-1}(\sym_+)$. It is well-known that $ \sym$ is an Euclidean space when equipped with the (\emph{Frobenius}) inner product $\langle M,N\rangle\doteq \tr(MN)\doteq \sum_{i,j=1}^m M_{ij} N_{ij}$.

The derivative of $G$ at a point $x\in \R^n$ is the linear mapping $DG(x)\colon\R^n\to \sym$ that can be described (in the canonical basis of $\R^n$) by the action
\[
	d\mapsto DG(x)[d]\doteq \sum_{i=1}^n D_{x_i}G(x)d_i
\]
for all $d=(d_1,\dots,d_n)\in \R^n$, where $D_{x_i}G(x)\in\sym$ is the partial derivative of $G$ with respect to the variable $x_i$ at $x=(x_1,\dots,x_n)\in\R^n$. Also, for each fixed $x$, the \emph{adjoint} of $DG(x)$  is the unique linear mapping $DG(x)^*\colon \sym\to \R^n$ that satisfies $\langle DG(x)[d],M \rangle=\langle d,DG(x)^*[M]\rangle,$ for all $(d,M)\in \R^n\times \sym$. Hence,
\[
	DG(x)^*[M]=
	\begin{bmatrix}
	\langle D_{x_1}G(x),M \rangle\\
	\vdots\\
	\langle D_{x_n}G(x),M \rangle
	\end{bmatrix}
%	=
%	\begin{bmatrix}
%	\sum_{i,j=1}^m D_{x_1}G_{ij}(x)M_{ij}\\
%	\vdots\\
%	\sum_{i,j=1}^m D_{x_n}G_{ij}(x)M_{ij}
%	\end{bmatrix}
%	=
%	\sum_{i,j=1}^m M_{ij}
%	\begin{bmatrix}
%	D_{x_1}G_{ij}(x)\\
%	\vdots\\
%	D_{x_n}G_{ij}(x)
%	\end{bmatrix}
	=
	\sum_{i,j=1}^m M_{ij}\nabla G_{ij}(x)
\]
for all $M\in \sym$, where $\nabla G_{ij}(x)$ denotes the \emph{gradient} of the $(i,j)$-th entry of $G$ as a function of $x$. Similarly, we shall denote the gradient of any real-valued function $F\colon \R^n\to \R$ at a point $x\in \R^n$ by $\nabla F(x)$.

For any given $M\in \sym$, we consider its spectral decomposition in the form
\begin{equation*}\label{specdecomp}
	\begin{aligned}
		M =\sum_{i=1}^m \lambda_i(M) u_i(M) u_i(M)^\T,
	\end{aligned}
\end{equation*}
where $\lambda_i(M)\in \R$ denotes the $i$-th eigenvalue of $M$ arranged in non-increasing order (that is, $\lambda_1(M)\geqslant \lambda_2(M)\geqslant \ldots\geqslant \lambda_m(M)$), and $u_i(M)\in \R^m$ corresponds to any associated eigenvector such that \si{the set} $\left\{u_i(M)\colon i\in \{1,\ldots,m\}\right\}$ is an orthonormal basis of $\R^m$ (that is, $u_i(M)^T u_i(M)=1$ and $u_i(M)^T u_j(M)=0$ when $i\neq j$, for all $i,j\in \{1,\ldots,m\}$).

A useful fact for our analyses is that the orthogonal projection of $M$ onto $\sym_+$ with respect to the induced (Frobenius) norm, denoted by $\projs(M)$, can be characterized in terms of its spectral decomposition as follows:
\[
	\projs(M)=\sum_{i=1}^m [\lambda_i(M)]_+ u_i(M) u_i(M)^\T,
\]
where $[\lambda]_+\doteq \max\{0,\lambda\}$ for all $\lambda\in \R$.

Given any $\bar{x}\in \F$ and any orthogonal matrix $\bar{U}\in \R^{m\times m}$ whose columns are eigenvectors of $G(\xb)$, we partition $\bar{U}=[\bar{P},\bar{E}]$ such that the columns of $\bar{P}\in \R^{m\times r}$ correspond to the eigenvectors associated with the positive eigenvalues of $G(\xb)$ and the columns of $\bar{E}\in \R^{m\times m-r}$ correspond to the eigenvectors associated with the null eigenvalues of $G(\xb)$, where $r=\rank(G(\xb))$. To abbreviate, as an abuse of notation and language, we will say that $\bar{E}$ \textit{spans} $\Ker G(\xb)$ in this context. That is, $\bar{E}$ spans $\Ker G(\xb)$ if, and only if, $\bar{E}^\T\bar{E}=\I_{m-r}$ and $G(\xb)\bar{E}=0$, where $\I_{m-r}$ denotes an $(m-r)$-dimensional identity matrix.

There are multiple ways of describing optimality in NSDP problems, but in this paper we direct our attention to necessary optimality conditions that are based on the classical \emph{Karush-Kuhn-Tucker} (KKT) conditions:

\begin{definition}[KKT]
	We say that a point $\xb\in \F$ satisfies the \emph{KKT conditions} when there exists some $\Yb\succeq 0$ such that
	\begin{equation}\label{def:kkt}
		\begin{aligned}
			\nabla_x L(\xb,\Yb)\doteq \nabla f(\xb)- D G(\xb)^*[\Yb]= 0,\\
			\langle G(\xb),\Yb\rangle = 0,
		\end{aligned}
		\tag{KKT}
	\end{equation}
	where $L:\R^n \times \sym  \to \R$ is the \emph{Lagrangian function} of \eqref{NSDP}, given by
	\begin{equation*}\label{eq:Lag_P}
		L(x,Y)\doteq f(x) - \langle G(x), Y\rangle.
	\end{equation*}
\end{definition}
As usual, the matrix $\Yb$ is called a \emph{Lagrange
multiplier} associated with  $\xb$ and we denote the set of all Lagrange multipliers associated with $\xb$ by
$\Lambda(\xb)$. When $\Lambda(\xb)\neq \emptyset$, $\xb$ is called a \emph{KKT point} of \eqref{NSDP}. Let $r$ be the rank of $G(\xb)$ and let
$\bar{E}\in \R^{m\times m-r}$ be a matrix that spans $\Ker G(\xb)$; then, for any $\Yb\in \Lambda(\xb)$, since both $\bar{Y}$ and $G(\xb)$ are positive semidefinite, the complementarity relation $\langle G(\xb), \Yb \rangle=0$ is equivalent to $ G(\xb) \Yb  =0$, which is in turn equivalent to saying that $\Im \Yb \subseteq (\Im G(\xb))^\perp = \Ker G(\xb)$, where $(\Im G(\xb))^\perp$ denotes the orthogonal complement of $\Im G(\xb)$.
Therefore, $\Yb$ is complementary to $G(\xb)$ if, and only if, it has the form
\begin{equation}	
	\label{def:bloques_Y}
	\Yb = \bar{E}\tilde{Y}\bar{E}^\top,
\end{equation}
where $\tilde{Y} \in \S^{m-r}_+$ is not necessarily a diagonal matrix. Moreover, note that $\tilde{Y}$ is not necessarily positive definite; that is, $\dim(\Ker \Yb)$ does not necessarily coincide with $r$. When they do coincide, $\xb$ and $\Yb$ are said to be \emph{strictly complementary}~\cite{Shapiro1997}.

It is known that the KKT conditions are not necessary for local optimality unless they are paired with a constraint qualification. For instance, one of the most studied constraint qualifications for \eqref{NSDP} is \emph{Robinson's CQ}~\cite[Def. 3]{Robinson1976}, which holds at a point $\xb\in \F$ if $\textnormal{there exists } d\in \R^n \textnormal{ such that }$ 
\[
	G(\xb)+ DG(\xb)[d]  \in \int\sym_+,
\]
where $\int \sym_+$ denotes the topological interior of $\sym_+$, which in turn coincides with the set of $m\times m$ symmetric positive definite matrices.
Alternatively, following Bonnans and Shapiro~\cite[Prop. 2.97]{bshapiro}, it is possible to say that (the dual form of) Robinson's CQ holds at $\bar{x}\in \F$ if, and only if, 
\begin{equation}\label{eq:robinsondual}
	\left.	
	\begin{aligned}	
		DG(\xb)^*[Y]=0\\
		\langle G(\xb),Y\rangle=0\\
		Y\succeq 0
	\end{aligned}
	\right\}
	\Rightarrow
	Y=0.
\end{equation}
Another well-known fact is that, for every local minimizer $\xb\in \F$, the set $\Lambda(\xb)$ is nonempty and compact if, and only if, Robinson's CQ holds at $\xb$ (see~\cite[Props. 3.9 and 3.17]{bshapiro} for details). This makes Robinson's CQ the natural analogue of the \emph{Mangasarian-Fromovitz CQ} (MFCQ), from NLP, in NSDP.

\section{The nondegeneracy condition for NSDP}\label{sec:ndg}

In this section, we discuss the well-known nondegeneracy condition introduced by Shapiro and Fan~\cite[Sec. 2]{shapfan}. %and later adapted by Shapiro~\cite[Def. 4]{Shapiro1997} for NSDP problems. 
We derive a different characterization for it that suggests a way of obtaining a weaker constraint qualification with potentially interesting properties. But firstly, we briefly recall some elements of convex analysis.

The (\emph{Bouligand}) tangent cone to 
a set $C$ at a point $y\in C$ is defined as
\[
	T_C(y)\doteq
	\left\{
	d \colon
		\begin{array}{l}
		\exists \seq{d}\to d, \ \exists \seq{t}\to 0, \ t^k>0,\\
		\forall k\in \N, \ y+t^k d^k\in C
		\end{array}
	\right\}.
\]
In particular, when $C=$
$\sym_+$, at a given $M\in \S^m_+$, it can be characterized as follows
\begin{equation*}
%\label{def:tangent_cone}
T_{\S^m_+}(M) = \left\{N\in \sym \colon d^\top N d \geqslant 0, \ \forall d\in \Ker M\right\}.
\end{equation*}
Therefore, for every feasible $\xb$ we have
\begin{equation}\label{eq:tangent_cone_SDP}
T_{\S^m_+}(G(\xb)) = \left\{N\in \sym\colon \bar{E}^\top N \bar{E} \succeq 0\right\},
\end{equation}
whenever $\bar{E}$ spans $\Ker G(\xb)$.

It is clear from \eqref{eq:tangent_cone_SDP} that the largest subspace contained in $T_{\S^m_+}(G(\xb))$, that is, its \emph{lineality space}, can be characterized as follows:
\begin{equation}\label{eq:lin_tangent_cone_SDP}
\lin(T_{\S^m_+}(G(\xb))) = \left\{N\in \sym\colon \bar{E}^\top N \bar{E}=0\right\}.
\end{equation}

The nondegeneracy condition of Shapiro and Fan is verified at $\xb$ when the linear subspaces $\Im DG(\xb)$ and $\lin(T_{\sym_+}(G(\xb)))$ of $\sym$ meet transversally, which is why it was originally called \emph{transversality} in~\cite{shapfan}. In mathematical language:
\begin{definition}[Def. 4 from~\cite{Shapiro1997}]\label{def:ndgclass}
	A point $\xb\in \F$ is said to satisfy the \emph{nondegeneracy} condition when the following relation is satisfied:
	\begin{equation}\label{def:NDG}
		\Im DG(\xb) + \lin(T_{\sym_+}(G(\xb)) ) = \sym.
	\end{equation}

\end{definition}

If $\xb$ is a local solution of \eqref{NSDP}, then nondegeneracy implies that $\Lambda(\xb)$ is a singleton; and the converse is also true in the presence of strict complementarity (see~\cite[Thm. 2.2 and Sect. 3]{Shapiro1997a}). Hence, Definition~\ref{def:ndgclass} is generally seen as an analogue of LICQ, from NLP, in NSDP. However, this analogy is tied to how the link between NLP and NSDP is made~\cite{Shapiro1997}. For example, when an NLP problem with constraints $g_1(x)\geqslant 0, \ldots, g_m(x)\geqslant 0$ is modelled as an NSDP with a single \emph{structurally diagonal} conic constraint; that is, with $G$ in the form
\begin{equation}\label{sdp:diagg}
	G(x)\doteq
	\begin{bmatrix}
		g_1(x) & &\\
		& \ddots &\\
		& & g_m(x)
	\end{bmatrix}\succeq 0;
\end{equation}
then Definition~\ref{def:ndgclass} fails whenever there is some $\bar{Y}\in\Lambda(\xb)$ and some nonzero $H\in \sym$ with only zeros in its diagonal, such that $H\succeq -\bar{Y}$, regardless of the linear independence of the set $\{\nabla g_1(\xb), \ldots, \nabla g_m(\xb)\}$. In fact, structurally diagonal NSDP problems are in general expected to lack uniqueness of the Lagrange multiplier.

On the other hand, it is well-known (cf. \cite[Sect. 4.6.1]{bshapiro}) that
a feasible point $\xb$ satisfies the nondegeneracy condition if, and only if, either $\Ker G(\xb)= \{0 \}$ or the linear mapping $\psi_{\xb}:\R^n \to \mathbb{S}^{m-r}$, defined by
\begin{equation}\label{def:psi_trans}
	\psi_{\xb}(d) \doteq \bar{E}^\top  D  G(\xb) [d] \bar{E},
\end{equation}
is surjective for any $\bar{E}$ that spans $\Ker G(\xb)$.
As a direct consequence of the equivalence above, it is possible to characterize  Definition~\ref{def:ndgclass} as follows:
\begin{proposition}[Prop. 6 from~\cite{Shapiro1997}]\label{prop:ndgvii}
Let $\xb\in \F$ and let $r$ denote the rank of $G(\xb)$. Then, $\xb$ satisfies the nondegeneracy condition  if, and only if, either $\Ker G(\xb)=\{0\}$ or the vectors 
\begin{equation}\label{def:vii}
	\begin{aligned}
		v_{ij}(\xb,\bar{E}) & \doteq \left[\bar{e}_i^\T D_{x_1}G(\xb) \bar{e}_j, \ldots, \bar{e}_i^\T D_{x_n}G(\xb) \bar{e}_j\right]^\T\\
		& = DG(\xb)^* \left[\frac{\bar{e}_i\bar{e}_j^\T+\bar{e}_j\bar{e}_i^\T}{2}\right],
	\end{aligned}
	\quad  1\leq i\leq j \leq m-r
\end{equation}
are linearly independent, where $\bar{E}\in\R^{m\times m-r}$ is an arbitrary fixed matrix that spans $\Ker G(\xb)$, and $\bar{e}_i$ denotes the $i$-th column of $\bar{E}$, for all $i\in \{1,\ldots,m-r\}$. 

\end{proposition}
Now, inspired by Proposition~\ref{prop:ndgvii}, we present a similar characterization of nondegeneracy that evaluates the linear independence of a narrower set of vectors at the cost of looking at all possible choices of $\bar{E}$ instead of a fixed one. Since our reasoning can be extended to Robinson's CQ, we also characterize it in a similar manner.

\begin{proposition}\label{sdp:ndg}
	Let $\xb\in \F$ and $r= \rank(G(\xb))$. Then, $\xb$ satisfies:
	\begin{enumerate}
	\item Nondegeneracy if, and only if, either $r=m$ or \jo{the set}
	\begin{equation}\label{ndg:theset}
		\left\{v_{ii}(\xb,\bar{E})\colon i\in\{1,\ldots,m-r\}\right\}
	\end{equation} is linearly independent for every matrix $\bar{E}\in\R^{m\times m-r}$ that spans $\Ker G(\xb)$.
	\item Robinson's CQ if, and only if, either $r=m$ or \eqref{ndg:theset} is positive linearly independent for every matrix $\bar{E}\in\R^{m\times m-r}$ that spans $\Ker G(\xb)$.
	\end{enumerate}
\end{proposition}

\begin{proof}
	Let us assume that $r<m$ since the result follows trivially otherwise. %If there exists some matrix $\bar{E}\in\R^{m\times m-r}$ that spans $\Ker G(\xb)$, such that \eqref{ndg:theset} is linearly dependent, then the larger set \si{$}\jo{\[}\left\{v_{ij}(\xb,\bar{E})\colon 1\leqslant i\leqslant j\leqslant m-r\right\}\si{$}\jo{\]} is also linearly dependent and it follows that $\xb$ violates nondegeneracy. 
	Then, for any fixed $\bar{E}\in \R^{m\times m-r}$ such that $G(\xb)\bar{E}=0$ and $\bar{E}^\T \bar{E}=\I_{m-r}$, note that \eqref{ndg:theset} is (positive) linearly independent if, and only if, the following holds: if the scalars $\alpha_1,\ldots,\alpha_{m-r}\in \R$ (with $\alpha_1\geqslant 0$, $\ldots$ , $\alpha_{m-r}\geqslant 0$, respectively) satisfy
\begin{equation}\label{sdp:eqrcq3}
	\sum_{i=1}^{m-r} \alpha_i DG(\xb)^*[\bar{e}_i\bar{e}_i^\T]
	%=
	%\sum_{i=1}^{m-r}\alpha_i \left[\bar{e}_i^\T D_{x_\ell}G(\xb)\bar{e}_i\right]_{\ell\in \{1,\ldots,n\}}
	=
	\sum_{i=1}^{m-r}\alpha_i v_{ii}(\xb,\bar{E})
	= 0,
\end{equation}
then one must have $\alpha_1=\ldots=\alpha_{m-r}=0$. That is, \eqref{ndg:theset} is (positive) linearly independent if, and only if, for every matrix $Y$ of the form
\begin{equation}\label{sdp:eqrcq1}
	Y \doteq 
	\sum_{i=1}^{m-r} \alpha_i \bar{e}_i \bar{e}_i^\T
	=
	\bar{E}
	\begin{bmatrix}
		\alpha_1 & &\\
		& \ddots  &\\
		& & \alpha_{m-r}\\
	\end{bmatrix}
	\bar{E}^\T
\end{equation}
where $\alpha_1,\ldots,\alpha_{m-r}\in \R$ (with $\alpha_1\geqslant 0$, $\ldots$ , $\alpha_{m-r}\geqslant 0$, respectively), we have that
\begin{equation}\label{sdp:eqrcq2}
	DG(\xb)^*[Y]=0 \ \Rightarrow \ Y=0.
\end{equation}
With this in mind, we recall that:
\begin{itemize}
	\item For any fixed choice of $\bar{E}$ spanning $\Ker G(\xb)$, nondegeneracy holds at $\xb$ if, and only if, \eqref{sdp:eqrcq2} holds for every $Y$ in the form $Y=\bar{E}Z\bar{E}^\T$ with $Z\in \mathbb{S}^{m-r}$ (Proposition~\ref{prop:ndgvii});
	\item Robinson's CQ holds at $\xb$ if, and only if, \eqref{sdp:eqrcq2} holds for every $Y\succeq 0$ such that $\langle G(\xb),Y \rangle=0$ (see \eqref{eq:robinsondual});
\end{itemize}
and it becomes clear that nondegeneracy (respectively, Robinson's CQ) implies that \eqref{ndg:theset} is (positive) linearly independent, for every $\bar{E}$ as described above, because every $Y$ as in~\eqref{sdp:eqrcq1} satisfies $\langle G(\xb),Y\rangle=0$.

To prove the converse of item 1, assume that \eqref{ndg:theset} is linearly independent for all $\bar{E}$ that spans $\Ker G(\xb)$. Let $Y=\bar{E}Z\bar{E}^\T$ be such that $Z\in \mathbb{S}^{m-r}$ and let $C\in \R^{m-r\times m-r}$ be an orthogonal matrix such that $C^\T Z C=\Diag(z_1,\ldots,z_{m-r})$, where $\Diag(z_1,\ldots,z_{m-r})\in \S^{m-r}$ is a diagonal matrix whose $i$-th diagonal entry is $z_i$, with $i\in \{1,\ldots,m-r\}$. Then, note that $\bar{E}C^\T$ also spans $\Ker G(\xb)$, which puts
\[
	Y=\bar{E}C^\T \Diag(z_1,\ldots,z_{m-r}) (\bar{E}C^\T)^\T
\] in format \eqref{sdp:eqrcq1}; by our previous assumption \eqref{sdp:eqrcq2} holds for $Y$ and we conclude that nondegeneracy holds at $\xb$.

Now, to prove the converse of item 2, assume that \eqref{ndg:theset} is positive linearly independent for all $\bar{E}$ that spans $\Ker G(\xb)$, and let $Y$ be such that $DG(\xb)^*[Y]=0$, $\langle G(\xb),Y\rangle=0$ and $Y\succeq 0$. It is elementary to see that there exists some matrix $\bar{E}$ that spans $\Ker G(\xb)$, such that $Y$ has the form \eqref{sdp:eqrcq1}. It follows from our hypothesis that $Y=0$ and because $Y$ is arbitrary, Robinson's CQ holds at $\xb$.
	\jo{\hfill\qed}
\end{proof}
The characterizations of nondegeneracy and Robinson's CQ from Proposition~\ref{sdp:ndg} may seem less practical than the one from Proposition~\ref{prop:ndgvii}, but it reveals a clear path for defining weaker CQs by ruling out some particular choices of $\bar{E}$, which is the main result of the next subsection.

We recall that Wachsmuth~\cite{Wachsmuth2013} proved for NLPs that LICQ is equivalent to the uniqueness of the Lagrange multiplier for any objective function $f$ (the unique multiplier may vary with $f$). Thanks to  Proposition~\ref{sdp:ndg} this characterization can be straightforwardly extended to NSDP replacing LICQ by nondegeneracy, which we omit.

\subsection{Sequences of eigenvectors and weak-nondegeneracy}

In~\cite{ahv}, Andreani et al. introduce a constructive technique for proving the existence of Lagrange multipliers for \eqref{NSDP}, which is based on the so-called \emph{sequential optimality conditions} from NLP~\cite{ahm10}. The core idea of their proof is to apply an external penalty algorithm to \eqref{NSDP} after regularizing it around a given local minimizer, to obtain a sequence of approximate KKT points converging to it, as follows:

\begin{theorem}[Thm. 3.2 from \cite{ahv}]\label{sdp:minakkt}
	Let $\xb$ be a local minimizer of \eqref{NSDP}. Then, for any sequence $\{\rho_k\}_{k\in \N}\to +\infty$, there exists some $\seq{x}\to \xb$, such that for every $k\in \N$, $x^k$ is a local minimizer of the regularized penalty function 
	\begin{equation*}\label{sdp:akktpen}
		f(x)+\frac{1}{2}\enorm{x-\xb}^2 
	  	+ \frac{\rho_k}{2} \norm{\Pi_{\sym_+}(-G(x))}^2.
	\end{equation*}
	In particular, computing derivatives we obtain $\nabla_x L(x^k,Y^k)\to 0$, where $Y^k \doteq \rho_k \Pi_{\sym_+}(-G(x^k))$.
\end{theorem}
With this result at hand, the authors prove that the sequence $\seq{Y}$ must be bounded in the presence of Robinson's CQ, and that all of its limit points are Lagrange multipliers associated with $\xb$~\cite[Thm. 6.1]{ahv}. Furthermore, the proof of this fact under nondegeneracy follows easily by contradiction: suppose that $\seq{Y}$ is unbounded, and take any limit point $\bar{Y}$ of the sequence $\left\{Y^k/\norm{Y^k}\right\}_{k\in \N}$; then:
\begin{enumerate}
\item It follows from $\nabla_x L(x^k,Y^k)\to 0$ that $DG(\xb)^*[\bar{Y}]=0$, which means $\bar{Y}\in \Ker DG(\xb)^*=\Im DG(\xb)^\perp$;
\item By the definition of $Y^k$, we have $0\neq\bar{Y}\succeq 0$ and $\langle G(\xb),\bar{Y}\rangle=0$, so $\bar{Y}\in \lin(T_{\S^m_+}(G(\xb)))^\perp$;
\end{enumerate}
Hence, $\bar{Y}\in \Im DG(\xb)^\perp \cap \lin(T_{\S^m_+}(G(\xb)))^\perp$, which contradicts nondegeneracy.

With a single extra step, which is to take a spectral decomposition of $Y^k$ for each $k$, the reasoning of the previous paragraph can be put in the same terms as Proposition~\ref{sdp:ndg}. Indeed,
%assuming for simplicity that  $\lambda_1(-G(x^k))\geqslant \ldots\geqslant \lambda_m(-G(x^k))$, for every $k$, 
observe that $\lambda_i(Y^k)=[\rho_k\lambda_i(-G(x^k))]_+=0$  for all $i\in \{m-r+1,\ldots,m\}$ and all $k$ large enough, because
\[
	\lambda_i(-G(x^k))=-\lambda_{m-i+1}(G(x^k)).
\]
So
\[
	\nabla_x L(x^k, Y^k) = \nabla f(x^k)- \sum_{i=1}^{m-r} [\rho_k\lambda_i(-G(x^k))]_+ v_{ii}(x^k,E^k)\to 0,
\]
where $E^k\in \R^{m\times m-r}$ is a matrix whose $i$-th column is $u_{m-i+1}(G(x^k))$. Then, note that if $E^k$ can be chosen such that at least one of its limit points $\bar{E}$ ensures linear independence of $\left\{v_{ii}(\xb,\bar{E})\colon i\in \{1,\ldots,m-r\}\right\}$, then $\seq{Y}$ must be bounded. Although the first clause of the previous sentence resembles nondegeneracy (as in Proposition~\ref{sdp:ndg}), note that asking for the linear independence of the set $\left\{v_{ii}(\xb,\bar{E})\colon i\in \{1,\ldots,m-r\}\right\}$ when $\bar{E}$ is not a limit point of some sequence $\seq{E}$ of eigenvectors of $G(x^k)$ seems unnecessary for defining a constraint qualification. This motivates us to propose a weaker variant of nondegeneracy in a way that can also be extended to Robinson's CQ, which goes as follows:
\begin{definition}[Weak-nondegeneracy and weak-Robinson's CQ]\label{sdp:licq}
	Let $\xb\in \F$ and let $r$ be the rank of $G(\xb)$. We say that \emph{weak-nondegeneracy} (respectively, \emph{weak-Robinson's CQ}) holds at $\xb$ if either $\Ker G(\xb)=\{0\}$ or: for every sequence $\seq{x}\to \xb$, there exists some sequence of matrices with orthonormal columns $\seq{E}\subseteq \R^{m\times m-r}$ such that:
	\begin{enumerate}
	\item The columns of $E^k$ are eigenvectors associated with the $m-r$ smallest eigenvalues of $G(x^k)$, for each $k\in \N$;
	\item There exists a limit point $\bar{E}$ of $\seq{E}$ such that the set \si{$}\jo{\[}\left\{v_{ii}(\xb,\bar{E})\colon i\in\{1,\ldots,m-r\}\right\},\si{$}\jo{\]} as defined in~\eqref{def:vii}, is (positive) linearly independent.
	\end{enumerate}
\end{definition}
There are a couple of nuances about Definition~\ref{sdp:licq} that should be properly addressed (see also the discussion after Remark~\ref{sdp:ndglicq}). First, we recall that the eigenvector functions $u_i(G(x))$, $i\in \{m-r+1,\ldots,m\}$ are not necessarily continuous at a given point $\xb$; so weak-nondegeneracy (and weak-Robinson's CQ) relies on the ``sequential continuity of eigenvectors'' along a given path. Second, for any fixed $\xb\in \F$ and any $\seq{x}\to \xb$, the sequence $\seq{E}$ described in Definition~\ref{sdp:licq} is well-defined for $k$ sufficiently large, since the $r$ largest eigenvalues of $G(x^k)$ are necessarily bounded away from zero. 

\begin{remark}\label{rem:limsup}
	Based on the previous discussion, it is worth mentioning that weak-nondegeneracy (and weak-Robinson's CQ) can be equivalently stated in terms of a certain notion of continuity of the eigenvectors of $G(x)$. To see why, consider a feasible point $\xb\in \F$ and let $r$ be the rank of $G(\xb)$. Because $r<m$, it follows that $\lambda_{r}(G(x))>\lambda_{r+1}(G(x))$ for every $x$ close enough to $\xb$, so the following set is well-defined:
\begin{equation}\label{def:er}
\mathcal{B}(x)\doteq \left\{  E\in \R^{m\times (m-r)}\colon \begin{array}{l} G(x)e_i=\lambda_{m-i+1}(G(x))e_i, \ \forall i\in\{1,\ldots,m-r\} \\ E^\T E=\I_{m-r}\end{array} \right\}
\end{equation}
where $E\doteq[e_1,\ldots,e_{m-r}]$. The set above consists of all matrices whose columns are orthonormal eigenvectors associated with the $m-r$ smallest eigenvalues of $G(x)$. Moreover, for any sequence $\seq{x}\to \xb$ recall the Painlev{\'e}-Kuratowski \textit{upper limit}~\cite[Def. 2.52]{bshapiro}) of the sequence of images $\{\mathcal{B}(x^k)\}_{k\in \N}$, defined as
\[
	\limsup_{k\to\infty} \mathcal{B}(x^k)\doteq \left\{z\colon \exists I\subseteq \N \textnormal{ infinite}, \ \exists \{z^k\}_{k\in I}\to z, \ \forall k\in I, \ z^k\in \mathcal{B}(x^k) \right\}.
\]
In these terms, it is easy to see that weak-nondegeneracy (respectively, weak-Robinson's CQ) holds at $\xb$ if, and only if, either $\Ker G(\xb)=\{0\}$ or, for every sequence $\seq{x}\to \xb$, there exists some $\bar{E}\in \limsup_{k\to\infty} \mathcal{B}(x^k)$ such that \si{$}\jo{\[}\left\{v_{ii}(\xb,\bar{E})\colon i\in\{1,\ldots,m-r\}\right\}\si{$}\jo{\]} is (positive) linearly independent. 
\end{remark}

Although the characterization of Remark~\ref{rem:limsup} may shorten notation, in order to check whether weak-nondegenearcy holds or not at a given point $\xb$ requires the computation of the set $\mathcal{B}(\xb)$, which may be complicated in practice. Therefore, it is important to emphasize that $\mathcal{B}(\xb)$ is not meant to be explicitly computed because weak-nondegeneracy is not meant to be manually checked at any point, except for very specific cases with a convenient eigenvector structure (see Examples \ref{ex:wndgnotndg} and \ref{ex:X}); instead, the main purpose of weak-nondegeneracy (and weak-Robinson's CQ) is to serve as a theoretical tool for building the convergence theory of iterative algorithms, as it was presented in the proof of Theorem 2 for the external penalty method. 
%
%In this context one's knowledge of the solution of the problem is usually limited to an approximation obtained by truncating the output sequence of the method, which ends up taking away some of the meaning of checking constraint qualifications in practice. 
%
In this context, knowledge of the problem solution is usually limited to an approximation obtained by truncating the method's output sequence, which ends up taking away some of the meaning of checking constraint qualifications in practice.

	The discussion that motivated Definition~\ref{sdp:licq} already suggests that it indeed describes a genuine constraint qualification, and it also provides an outline of how to prove it. Nevertheless, we state and prove this fact with appropriate mathematical rigor below. Although we prove the next result for weak-Robinson's CQ, observe that the analogous statement for weak-nondegeneracy follows as a corollary.
\begin{theorem}\label{thm:wndgcq}
Every local minimizer $\xb\in \F$ of \eqref{NSDP} that satisfies weak-Robinson's CQ also satisfies the KKT conditions. By extension, the same holds for weak-nondegeneracy.
\end{theorem}
\begin{proof}
Let $\bar{x}$ be a local minimizer of \eqref{NSDP} that satisfies weak-Robinson's CQ and let $\seq{x}\to \bar{x}$ and $\seq{Y}$ be the sequences described in Theorem~\ref{sdp:minakkt}, for an arbitrary sequence $\{\rho_k\}_{k\in \N}\to \infty$. If $r=m$, set $\bar{Y}=0$ as a Lagrange multiplier associated with $\xb$ and we are done; so let us assume that $r<m$ from now on. From the local optimality of $x^k$, for each $k\in \N$, we obtain
\begin{equation}\label{sdp:akktstat}
	\nabla f(x^k)+(x^k-\xb)+DG(x^k)^*[Y^k]=0.
\end{equation}
Recall that we assume, without loss of generality, that $\lambda_1(-G(x^k))\geqslant \ldots\geqslant \lambda_m(-G(x^k))$, for every $k$; and note that when $k$ is large enough, say greater than some $k_0\in \N$, we necessarily have $\lambda_i(-G(x^k))< 0$ for all $i\in \{m-r+1,\ldots,m\}$ since $G(x^k)\to G(\xb)$ and eigenvalues $\lambda_i(\cdot)$ are continuous mappings.  Then, for each $k>k_0$, we have
		\[
			Y^k=\sum_{i=1}^{m-r} \alpha^k_i e_i^k(e_i^k)^\T,
		\]	
	where $\alpha_i^k\doteq [\rho_k\lambda_i(-G(x^k))]_+\geqslant 0$ 
	%(is nonnegative for all $i\in \{1,\ldots,m-r\}$ and all $k\in \N$)
	 and $e_i^k\doteq u_{m-i+1}(G(x^k))$ is an arbitrary unitary eigenvector associated with $\lambda_{m-i+1}(G(x^k))$, for each $i\in \{1,\ldots,m-r\}$. Set $E^k\doteq[e_1^k,\ldots,e_{m-r}^k]$. Since $\seq{E}$ is bounded, we may pick any of its limit points $\bar{E}=[\bar{e}_1,\ldots,\bar{e}_{m-r}]$ and assume, taking a subsequence if necessary, that it converges to $\bar{E}$, which spans $\Ker G(\xb)$. Then, observe that \eqref{sdp:akktstat} implies
	\begin{equation*}\label{sdp:decompstat}
		\nabla f(x^k)-\sum_{i=1}^{m-r} \alpha^k_i DG(x^k)^*[e_i^k(e_i^k)^\T]\to 0,
	\end{equation*}
	but since $DG(x^k)^*[e_i^k (e_i^k)^\top]=v_{ii}(x^k,E^k)$ (see~\eqref{def:vii}),
%	\[	
%		DG(x^k)^*[e_i^k (e_i^k)^\top]=
%		\begin{bmatrix}
%			\langle D_1 G(x^k), e_i^k (e_i^k)^\top \rangle\\
%			\vdots\\
%			\langle D_n G(x^k), e_i^k (e_i^k)^\top \rangle
%		\end{bmatrix}
%		=
%		\begin{bmatrix}
%			(e_i^k)^\top D_1 G(x^k) e_i^k\\
%			\vdots\\
%			(e_i^k)^\top D_n G(x^k) e_i^k
%		\end{bmatrix}
%		=
%		v_{ii}(x^k,E^k),
%	\]
	we can rewrite it as
	\begin{equation}\label{sdp:akktvii}
		\nabla f(x^k)-\sum_{i=1}^{m-r} \alpha^k_i v_{ii}(x^k,E^k)\to 0.
	\end{equation}
	
	If $\left\{(\alpha_i^k,\ldots,\alpha_{m-r}^k)\right\}_{k\in \N}$ has any convergent subsequence, denote its limit point by $\bar{\alpha}\doteq(\bar\alpha_i,\ldots,\bar\alpha_{m-r})$, and note that $\bar{\alpha}$ generates a Lagrange multiplier for $\xb$, which is
	\begin{equation}\label{weak:multiplier}
		\bar{Y}\doteq \sum_{i=1}^{m-r}\bar{\alpha}_i\bar{e}_i\bar{e}_i^\T.
	\end{equation}
	Hence, it suffices to prove that $\seq{\alpha_i}$, $i\in \{1,\ldots,m-r\}$, must be bounded under weak-Robinson's CQ. Let us assume for a moment that the sequences $\seq{\alpha_i}$ are unbounded, which means
	\[
		m^k\doteq \max\left\{\alpha_i^k\colon i\in \{1,\ldots,m-r\}\right\}\to \infty.
	\]
	Note that $\left\{(\alpha^k_1,\ldots,\alpha^k_{m-r})/m^k\right\}_{k\in \N}$ must be bounded and it must also have a nonzero limit point, which we will denote by $(\tilde{\alpha}_1,\ldots,\tilde{\alpha}_{m-r})\geqslant 0$. We assume without loss of generality, that $\left\{(\alpha^k_1,\ldots,\alpha^k_{m-r})/m^k\right\}_{k\in \N}\to (\tilde{\alpha}_1,\ldots,\tilde{\alpha}_{m-r})$. After dividing \eqref{sdp:akktvii} by $m^k$ for each $k$ and taking limit $k\to +\infty$, we obtain 
	\[
		\sum_{i=1}^{m-r}\tilde{\alpha}_iv_{ii}(\xb,\bar{E})=0,
	\]
	which means $\left\{v_{ii}(\xb,\bar{E})\colon i\in \{1,\ldots,m-r\}\right\}$ is positive linearly dependent. However, since our analyses hold for any arbitrary choice of $\seq{E}$ and any $\bar{E}$, this contradicts weak-Robinson's CQ.
		\jo{\hfill\qed}
\end{proof}

Let us briefly analyse a direct application of weak-Robinson's CQ: As an intermediary step of the proof of Theorem~\ref{thm:wndgcq}, we proved that every feasible limit point of a sequence described in Theorem~\ref{sdp:minakkt} must satisfy the KKT conditions under weak-Robinson's CQ. However, the sequences $\seq{x}$ and $\seq{Y}$ described in Theorem~\ref{thm:wndgcq} are precisely the ones that are generated by a standard external penalty method (that is, \cite[Algorithm 1]{ahv} with the parameter $\Omega^{\max}$ fixed at zero). Thus, every feasible limit point of the external penalty method that satisfies weak-Robinson's CQ must also satisfy the KKT conditions. By extension this also holds for weak-nondegeneracy.

Another interesting property of the weak variants of nondegeneracy and Robinson's CQ is that they are equivalent to LICQ and MFCQ, respectively, when $G$ is a structurally diagonal matrix function (as in \eqref{sdp:diagg}) that models an NLP problem, which in some sense resolves the inconsistency between nondegeneracy and LICQ noted by Shapiro~\cite[Page 309]{Shapiro1997}.
\begin{remark}\label{sdp:ndglicq}
	When $G$ is structurally diagonal, as in \eqref{sdp:diagg}, then $\xb\in \F$ satisfies weak-nondegeneracy if, and only if, the set $\{\nabla g_i(\xb)\colon g_i(\xb)=0\}$ is linearly independent. Indeed, if $r=m$ the result follows trivially, so let us assume that $r<m$. Also, suppose that $\{i\in \{1,\ldots,m\}\colon g_i(\xb)=0\}=\{r+1,\ldots,m\}$, where $r$ is the rank of $G(\xb)$. Clearly, if $\{\nabla g_{r+1}(\xb),\ldots,\nabla g_{m}(\xb)\}$ is linearly independent, then we may take 
	\[
		E^k\doteq
		\begin{bmatrix}
			0\\
			\I_{m-r}
		\end{bmatrix}\in \R^{m\times m-r}
	\]
for all sequences $\seq{x}\to \xb$ to conclude that $\xb$ satisfies weak-nondegeneracy. Conversely, suppose that weak-nondegeneracy holds at $\xb$, take any sequence $\seq{x}\to \xb$ and any $\seq{E}\to \bar{E}\doteq [\bar{e}_i,\ldots,\bar{e}_{m-r}]$ such that \si{$}\jo{\[}\left\{v_{ii}(\xb,\bar{E})\colon i\in\{1,\ldots,m-r\}\right\}\si{$}\jo{\]} is linearly independent. Note that $\bar{E}$ must have the form
	\[
		\bar{E}=
		\begin{bmatrix}
			0\\
			Q
		\end{bmatrix},
		\
		\textnormal{where $Q\in \R^{m-r\times m-r}$ is orthonormal},
	\]
	due to the diagonal structure of $G$ and the fact that $g_{i}(\bar{x})\neq 0$ for all $i\in \{1,\ldots,r\}$. Hence,
	\begin{equation}\label{eq:diagvii}
		v_{ii}(\bar{x},\bar{E})=\sum_{j=r+1}^{m} \nabla g_j(\bar{x})Q_{i,j-r}^2=Dg(\xb)^\T (Q_i\odot Q_i),  
	\end{equation}
	where $Dg(\xb)$ is the Jacobian matrix of $g(x)\doteq(g_{r+1}(x),\ldots,g_{m}(x))$ at $\xb$, the operator $\odot$ is the (Hadamard) entry-wise vector product, and $Q_i$ is the $i$-th column of $Q$, with $i\in \{1,\ldots,m-r\}$. Then,
	\[
		\textnormal{span}\left\{v_{ii}(\xb,\bar{E})\colon i\in\{1,\ldots,m-r\}\right\}\subseteq \Im Dg(\xb)^\T
	\] 
	and, consequently,
	\[
		\begin{aligned}
		m-r & =\dim(\textnormal{span}\left\{v_{ii}(\xb,\bar{E})\colon i\in\{1,\ldots,m-r\}\right\}) & \\
		& \leqslant
		\dim(\Im Dg(\xb)^\T) & \\
		& =
		\rank (Dg(\xb)^\T)\leqslant m-r
		\end{aligned}
	\]
	Hence, $\rank (Dg(\xb)^\T) = m-r$, which means that $\{\nabla g_{r+1}(\xb),\ldots,\nabla g_{m}(\xb)\}$ is linearly independent. Using similar arguments, thanks to~\eqref{eq:diagvii} which states that the vectors $v_{ii}(\bar{x},\bar{E})$ are nonnegative linear combinations of the columns of $Dg(\xb)^\T$, it is possible to prove that $\xb\in \F$ satisfies weak-Robinson's CQ if, and only if, $\{\nabla g_i(\xb)\colon g_i(\xb)=0\}$ is positive linearly independent, which is in turn equivalent to Robinson's CQ.
\end{remark}

It is clear from~Proposition~\ref{sdp:ndg} that weak-nondegeneracy is implied by nondegeneracy; and we see in the example below that the converse is not true.

%In light of Remark~\ref{sdp:ndglicq}, let us go back to Definition~\ref{sdp:licq} for a moment, to clarify that weak-nondegeneracy is not equivalent to requiring linear independence of $\left\{v_{ii}(\xb,\bar{E})\colon i\in \{1,\ldots,m-r\}\right\}$ for all $\bar{E}$ in the set
%\[
%	\mathcal{B}(\xb)\doteq 
%	\left\{ \bar{E}\in \R^{m\times m-r}\colon  
%	\jo{\begin{array}{l}}
%	\exists \seq{x}\to \xb, \exists \seq{E}\to \bar{E}, \jo{ \\ }\si{ \ } x^k\neq \xb, \
%	 \textnormal{item 1 of Def.~\ref{sdp:licq} holds}
%	 \jo{\end{array}}
%	 \right\}.
%\]
%That is, we do not require linear independence of $\{v_{ii}(\xb,\bar{E})\colon i\in\{1,\dots,m-r\}\}$ for all matrices $\bar{E}$ that are limits of matrices $E^k$ of eigenvectors of $G(x^k)$. This is relevant when the eigenvalues of $G(x^k)$ are non-simple, and thus $E^k$ is not uniquely defined.
%The following example illustrates that. 
%
\begin{example}\label{ex:wndgnotndg}
Consider the following constraint:
	\begin{equation*}
		G(x)\doteq
		\begin{bmatrix}
			x_1 & x_2\\
			x_2 & x_1
		\end{bmatrix}
	\end{equation*}
	at the point $\xb\doteq (0,0)$, which clearly does not satisfy nondegeneracy. Weak-nondegeneracy, on the other hand, holds at $\xb$ as
	\[
		\mathcal{B}(x) = 
		\left\{
			\begin{array}{ll}
			\left\{
				\frac{\pm 1}{\sqrt{2}}
				\begin{bmatrix}
					-1 & 1\\
					 1 & 1
				\end{bmatrix},
				\frac{\pm 1}{\sqrt{2}}
				\begin{bmatrix}
					1 & 1\\
					-1 & 1
				\end{bmatrix}
			\right\}, & \textnormal{ if }  x_2\neq 0\\
			\\
			\left\{
				E\in \R^{2\times 2} \colon E^\T E=\I_2
			\right\}, & \textnormal{ if }  x_2=0
			\end{array}
		\right.
	\]
	for every $x\in \R^2$, according to \eqref{def:er}, so it suffices to take
	\[
		E^k\doteq
		\frac{1}{\sqrt{2}}
		\begin{bmatrix}
			-1 & 1\\
			1 & 1
		\end{bmatrix}\in \mathcal{B}(x^k)
		\quad \textnormal{and} \quad
		\bar{E}\doteq
		\frac{1}{\sqrt{2}}
		\begin{bmatrix}
			-1 & 1\\
			1 & 1
		\end{bmatrix}\in\limsup_{k\to\infty} \mathcal{B}(x^k)
	\]
	for all sequences $\seq{x}\to \xb$ to obtain that $v_{11}(\xb,\bar{E})=[1,-1]$ and $v_{22}(\xb,\bar{E})=[1,1]$ are linearly independent.
	%However, $\xb$ satisfies weak-nondegeneracy due to Remark~\ref{sdp:ndglicq}. In fact, given any sequence $\seq{x}\to \xb$, it suffices to pick $E^k\doteq \I_2$ for all $k\in \N$ to verify that weak-nondegeneracy holds at $\xb$. 

This simple example is also important to show that weak-nondegeneracy does not guarantee uniqueness of Lagrange multipliers. For instance, consider the constraint above with the objective function $f(x)\doteq 2x_1$ which has $\xb$ as its global minimizer; then every $\bar{Y}$ in the form
\[
	\bar{Y}\doteq \begin{bmatrix}
		1 -\alpha & 0\\
		0 & 1+\alpha
	\end{bmatrix}
\]
for $\alpha\in [-1,1]\setminus \{0\}$ is a Lagrange multiplier associated with $\xb$.
\end{example}

Another example that serves the same purpose, which can also be used to show how the sparsity structure of the eigenvectors of $G$ is grasped by weak-nondegeneracy is the following:

\begin{example}\label{ex:X}
Consider the constraint:
\[
	G(x)\doteq \begin{bmatrix}
	x_{11} & 0 & x_{13}\\
	0 & x_{22} & 0\\
	x_{13} & 0 & x_{33}
	\end{bmatrix}\succeq 0
\]
and let $\xb\doteq 0$. Nondegeneracy fails at $\xb$, but weak-nondegeneracy holds. To see this, take any sequence $\seq{x}\to \xb$, if $x_{13}^k\neq 0$ for all $k$ (the other case is trivial, so we will omit it), and
\[
	E^k\doteq \begin{bmatrix}
	\frac{-\eta^k_1}{\sqrt{(\eta^k_1)^2+1}} & 0 & \frac{-\eta^k_2}{\sqrt{(\eta^k_2)^2+1}}\\
	0 & 1 & 0\\
	\frac{1}{\sqrt{(\eta^k_1)^2+1}} & 0 & \frac{1}{\sqrt{(\eta^k_2)^2+1}}
	\end{bmatrix},
\] 
where
\[
	\eta_1^k\doteq \frac{-x_{11}^k+x_{22}^k+\sqrt{(x_{11}^k)^2-2x_{22}^k x_{33}^k+(x_{33}^k)^2+4(x_{13}^k)^2}}{2x_{13}^k}
\]
and
\[
	\eta_2^k\doteq \frac{-x_{11}^k+x_{22}^k-\sqrt{(x_{11}^k)^2-2x_{22}^k x_{33}^k+(x_{33}^k)^2+4(x_{13}^k)^2}}{2x_{13}^k}.
\]
In this case, assuming that $x_{13}^k>0$ for all $k$ (which can be done without loss of generality since the other cases are analogous), we have 
\[
	\lim_{k\to \infty} \eta_1^k=\lim_{k\to \infty} \frac{|x_{13}^k|}{x_{13}^k}=1 \quad \textnormal{and} \quad \lim_{k\to \infty} \eta_2^k=\lim_{k\to \infty} -\frac{|x_{13}^k|}{x_{13}^k}=-1,
\] hence
\[
	E^k\to \bar{E}\doteq  \begin{bmatrix}
	\frac{-1}{\sqrt{2}} & 0 & \frac{1}{\sqrt{2}}\\
	0 & 1 & 0\\
	\frac{1}{\sqrt{2}} & 0 & \frac{1}{\sqrt{2}}
	\end{bmatrix},
\]
and computing the vectors of interest we arrive at
\[
	v_{11}(\xb,\bar{E})=\frac{1}{2}\begin{bmatrix}
	1\\
	0\\
	1\\
	-2
	\end{bmatrix},
	\quad
	v_{22}(\xb,\bar{E})=\begin{bmatrix}
	0\\
	1\\
	0\\
	0
	\end{bmatrix}.
	\quad
	v_{33}(\xb,\bar{E})=\frac{1}{2}\begin{bmatrix}
	1\\
	0\\
	1\\
	2
	\end{bmatrix},
\]
which are linearly independent, so weak-nondegeneracy holds at $\xb$. Observe that, in this case, the matrix $E^k$ has the same sparsity structure as $G$.
\end{example}

Moreover, note that weak-nondegeneracy imposes a less demanding dimensionality constraint over \eqref{NSDP}; in fact, in order to verify nondegeneracy, one must have $n\geqslant (m-r)(m-r+1)/2$, while weak-nondegeneracy may hold as long as $n\geqslant m-r$ (Remark~\ref{sdp:ndglicq}). It is also clear from their definitions that weak-nondegeneracy implies weak-Robinson's CQ; and it is possible to show that the converse is not necessarily true. For instance, consider the constraint defined by:
\begin{equation*}
	G(x)\doteq
	\begin{bmatrix}
		x & 0\\
		0 & x
	\end{bmatrix},
\end{equation*}
and note that all orthogonal matrices $\bar{E}\in \R^{2\times 2}$ have in their columns eigenvectors of $G(x)$, for every $x$. Since $v_{11}(\xb,\bar{E})=v_{22}(\xb,\bar{E})=1$ for every $\bar{E}$, it follows that weak-nondegeneracy and weak-Robinson's CQ are equivalent to their strong counterparts in this case. Thus, from Proposition~\ref{sdp:ndg} we see that (weak-)nondegeneracy does not hold, while (weak-)Robinson's CQ does.

It is also clear from Proposition~\ref{sdp:ndg} (item 2) that Robinson's CQ implies weak-Robinson's CQ; however, we were not capable of finding a counterexample for the converse. We conjecture that they are equivalent.

\begin{remark}
If we replace the sequences $\seq{x}\to \xb$ by matrix sequences $\seq{M}\to G(\xb)$ in Definition~\ref{sdp:licq}, then we recover the nondegeneracy condition. Indeed,  for any $\bar E\in\R^{m\times m-r}$ that spans $\Ker G(\xb)$, consider 
$$M^k\doteq \bar U \Lambda^k \bar U^\top, \  \mbox{ with  } \  \bar U \doteq [\bar E , u_{m-r+1}(G(\bar x)),\ldots, u_{m}(G(\bar x))],$$
and $  \Lambda^k \doteq\Diag(y^k)$ such that $y^{k}_{i}\doteq i/k  $ for $i\in \{1,\ldots,m-r\}$, and  $y^{k}_{i}\doteq   \lambda_i(G(\bar x)) $ otherwise. So, clearly $M^k\to G(\bar x)$ and the only convergent sequence $E^k$ to $\bar E$ is $\bar E$ itself. Consequently, when we assume Definition~\ref{sdp:licq} it necessarily follows that $\{v_{ii}(\bar x, \bar E)\colon i\in \{1,\ldots,m-r\}\}$ is linearly independent. Then, since $\bar E$ was chosen arbitrary, Proposition~\ref{sdp:ndg} implies that nondegeneracy holds true.
\end{remark}

\begin{remark}\label{rem:wndgblockdiag}
Remark~\ref{sdp:ndglicq} can be straightforwardly extended to structurally block diagonal matrix constraints, such as
\begin{equation}
  \tag{Block-NSDP}
  \begin{aligned}
%    & \underset{x \in \mathbb{R}^{n}}{\text{Minimize}}
%    & & f(x), \\
%    & \text{subject to}
%    %& & h(x)=0,\\
    & & G(x)\doteq 
	\begin{bmatrix}
		G_1(x) & & \\
		& \ddots & \\
		& & G_q(x)
	\end{bmatrix}\succeq 0,
  \end{aligned}
  \label{sdp:blockdiag}
\end{equation}
where each ``block'' is defined by a continuously differentiable function $G_{\ell}\colon \R^n\to \S^{m_\ell}$, with $\ell\in \{1,\ldots,q\}$, and $m_1+\ldots+m_q=m$. In fact, let $\xb\in \F$ and $r_\ell\doteq \rank(G_\ell(\xb))$ for each $\ell$; and, for simplicity, let us assume that $r_\ell<m_\ell$ for all $\ell$.
Since $\Ker G(\xb)=\Ker G_1(\xb)\times\ldots\times \Ker G_q(\xb)$, then $r=r_1+\ldots+r_q$.  Then, weak-nondegeneracy (respectively, weak-Robinson's CQ) holds at $\xb$ if, and only if, for all sequences $\seq{x}\to \xb$, there are sequences of matrices $\{E^k_\ell\}_{k\in \N}$ such that:
\begin{itemize}
\item The columns of $E^k_\ell$ are unitary eigenvectors associated with the $m_\ell-r_\ell$ smallest eigenvalues of $G_\ell(x^k)$, for each $k\in \N$ and each $\ell\in \{1,\ldots,q\}$;
\item There are limit points $\bar{E}_\ell$ of $\{E^k_\ell\}_{k\in \N}$, $\ell\in\{1,\ldots,q\}$, such that the set
\[
	\bigcup_{\ell=1}^q \left\{ v_{ii}^{\ell}(\xb,\bar{E}_\ell) \colon i\in \{1,\ldots,m_\ell-r_\ell\} \right\}
\]
is (positive) linearly independent, where
\begin{equation}\label{block:viidef}
	v_{ij}^\ell(\xb,\bar{E}_\ell)\doteq \left[\bar{e}_{\ell,i}^\T D_{x_1}G_\ell(\xb) \bar{e}_{\ell,j} \ , \ \ldots \ , \ \bar{e}_{\ell,i}^\T D_{x_n}G_\ell(\xb) \bar{e}_{\ell,j}\right]^\T,
\end{equation}
and $\bar{e}_{\ell,1},\ldots,\bar{e}_{\ell,m_\ell-r_\ell}$ denote the columns of $\bar{E}_\ell$, for each $\ell\in \{1,\ldots,q\}$.
\end{itemize}
The proof of this fact is elementary with~\cite[Lem. 1.3.10]{hornjohnson} and~\eqref{def:vii} at hand. Moreover, note that this is precisely the way weak-nondegeneracy would be defined for an equivalent multifold NSDP with constraints $G_1(x)\succeq 0, \ldots, G_q(x)\succeq 0$. Thus, weak-nondegeneracy and weak-Robinson's CQ are invariant to block diagonal and multifold representations of \eqref{sdp:blockdiag}. This is especially meaningful in problems that do not present an explicit block-diagonal representation, in which case it is not necessary to have prior knowledge of such a representation to talk about weak-nondegeneracy (or weak-Robinson's CQ), contrary to nondegeneracy.
\end{remark}

Recall that the analysis we presented until this point showed, among other things, that some choices of $\bar{E}$ may be more meaningful than others. With this in mind, we are now led to revisit the work of Forsgren~\cite{Forsgren2000}, who presented a very interesting way of talking about nondegeneracy in the presence of any sparsity structure that appears after applying a particular transformation to the problem. In the next section, we improve some of Forsgren's ideas by presenting a simplified and more general way of dealing with sparsity.

\section{Dealing with structural sparsity}\label{sec:sndg}

In this section, we take inspiration from a regularity condition introduced by Forsgren~\cite[Sect. 2.3]{Forsgren2000}, whose primary goal was to prove second-order optimality conditions for \eqref{NSDP}. However, what makes Forsgren's condition specially interesting for us is the fact it can benefit from some sparsity structure of a certain Schur complement related to the constraint function. The main objective of this section is to present a more straightforward way of enjoying sparsity, based on Forsgren's results and Section~\ref{sec:ndg}. But before that, we present some of the notation used by Forsgren.

Given a point $\xb$ and a matrix-valued function $F\colon \R^n\to \S^\beta$, consider the set $\mathcal{S}(F,\bar{x})$ defined as follows: 
\[
	\begin{array}{ll}
		\mathcal{S}(F,\xb)&\doteq\left\{M\in \S^\beta\colon M_{ij}=0 \text{ if } F_{ij}(x) \text{ is structurally zero near }\xb\right\}\\
		& =\left\{M\in \S^\beta\colon M_{ij}=0 \text{ if } \exists \varepsilon>0 \text{ such that } F_{ij}(x)=0,\forall x\in B(\xb,\varepsilon)\right\}.
	\end{array}
\]
For example, if $\beta=3$ and for all $x$ close to $\xb$, we are able to identify non trivial mappings $F_{ij}$ such that
\begin{equation*}
	F(x)=
	\begin{bmatrix}
		F_{11}(x) & 0 & F_{13}(x)\\
		0 & F_{22}(x) & 0\\
		F_{13}(x) & 0 & F_{33}(x)
	\end{bmatrix},
	\textnormal{ then } 
	M\in \mathcal{S}(F,\xb)
	\Leftrightarrow
	M=
	\begin{bmatrix}
		M_{11} & 0 & M_{13}\\
		0 & M_{22} & 0\\
		M_{13} & 0 & M_{33}
	\end{bmatrix},
\end{equation*}
where $M_{11}, M_{13}, M_{22},$ and $M_{33}$ may or may not be zero.
Also, we define
\[
	\mathcal{I}(F,\xb)\doteq
	\left\{(i,j) \colon \forall \varepsilon>0, \exists x\in B(\xb,\varepsilon) \text{ such that } F_{ij}(x)\neq 0, \ 1\leqslant i\leqslant j\leqslant \beta\right\}
\]
as the set of indices that define the elements of $\mathcal{S}(F,\xb)$. 

Forsgren's results are obtained in terms of the function
\[
	\tilde{G}(x)\doteq G(x)-G(x)\bar{P}(\bar{P}^\top G(x)\bar{P})^{-1}\bar{P}^\top G(x),
\]
where $\bar{U}=[\bar{P},\bar{E}]$ has columns that form an orthonormal eigenvector basis for $G(\xb)$, such that $\bar{E}$ spans the kernel of $G(\xb)$ and $\bar{P}^\T G(\xb)\bar{P}\succ 0.$ Note that $\bar{E}^\top \tilde{G}(x)\bar{E}$ is the Schur complement of $\bar{P}^\top{G}(x)\bar{P}$ inside
\[
	\bar{U}^\T G(x) \bar{U}=
	\begin{bmatrix}
		\bar{P}^\T G(x)\bar{P} & \bar{P}^\T G(x)\bar{E}\\
		\bar{E}^\T G(x)\bar{P} & \bar{E}^\T G(x)\bar{E}
	\end{bmatrix}.
\]	
Moreover, following Forsgren~\cite[Lem. 1]{Forsgren2000}, we see that $\tilde{G}(x)\succeq 0$ if, and only if $G(x)\succeq 0$, for all $x$ sufficiently close to $\xb$, so the original NSDP problem can be locally reformulated as a minimization problem over $\tilde{G}(x)\succeq 0$, around $\xb$. In fact, since
\begin{align*}
	\bar{P}(\bar{P}^\top G(\xb)\bar{P})^{-1}\bar{P}^\top & =\bar{P}\lambda_+(G(\xb))^{-1}\bar{P}^\T \\
	& = \bar{U}
	\begin{bmatrix}
		\lambda_+(G(\xb))^{-1} & 0\\
		0 & 0
	\end{bmatrix}		
	\bar{U}^\T\\
	& =G(\xb)^{\dag},
\end{align*}
where $G(\xb)^{\dag}$ is the Moore-Penrose pseudoinverse of $G(\xb)$, it follows that $\tilde{G}(\xb)=0$ \cite[Lem. 2]{Forsgren2000}, so $\tilde{G}$ can be considered a reduction to the kernel of $G(\xb)$ near $\xb$. 

The regularity condition introduced by Forsgren is as follows:
\begin{definition}[Forsgren's CQ]
	Let $\xb\in \F$ and let $\bar{U}\doteq[\bar{P},\bar{E}]$ be an orthogonal matrix that diagonalizes $G(\xb)$, such that the columns of $\bar{E}$ span $\Ker G(\xb)$. Then, \emph{Forsgren's CQ} holds at $\xb$ with respect to $\bar{U}$ when
	\begin{equation}\label{eq:forsgren1}
		\textnormal{span}\left\{\bar{E}^\top D_{x_i}G(\xb)\bar{E} \colon i\in \{1,\dots,n\}\right\}=\bar{E}^\top \mathcal{S}(\tilde{G},\xb) \bar{E}
		\tag{F1}
	\end{equation}
	and
	\begin{equation}\label{eq:forsgren2}
	\exists M\in \bar{E}^\top \mathcal{S}(\tilde{G},\xb) \bar{E}, \ \textnormal{ such that } \ M\succ 0.
	\tag{F2}
	\end{equation}
\end{definition}
Forsgren's CQ is indeed a constraint qualification, for when \eqref{eq:forsgren1} holds, then \eqref{eq:forsgren2} is equivalent to Robinson's CQ~\cite[Lem. 5]{Forsgren2000}. However, although Forsgren states that any choice of $\bar{U}$ leads to a valid CQ, there is no discussion on the effects of this choice over the condition proposed. Under a specific condition, Forsgren's CQ provides uniqueness of the Lagrange multiplier~\cite[Thm. 1]{Forsgren2000}, but this condition varies with $\bar{U}$. Thus, different choices of $\bar{U}$ are likely to define different variants of Forsgren's CQ. This is not necessarily a negative point, but a comparison among those variants would be appropriate. For instance, from the practical point of view, one may be interested in knowing which choice of $\bar{U}$ defines the weakest CQ, or which one is easier to compute.

A result from Dorsch, G{\'o}mez, and Shikhman~\cite{Dorsch2016} shows that, ignoring the sparsity treatment, \eqref{eq:forsgren1} becomes equivalent to nondegeneracy.
\begin{lemma}[Lem. 5 from~\cite{Dorsch2016}]\label{lem:fcqndg}
	Let $\xb\in \F$ and assume that $\mathcal{S}(\tilde{G},\xb)=\sym$. Then, condition \eqref{eq:forsgren1} of Forsgren's CQ holds if, and only if, nondegeneracy holds at $\xb$.
\end{lemma}
%
%\begin{proof}
%	Define the linear operator $\psi:\R^n\to \S^{m-r}$ by the action $\psi(d)\doteq \bar{E}^\top DG(\xb)[d]\bar{E}$ and note that 
%	\[
%		\Im(\psi)\doteq\textnormal{span}\left\{\bar{E}^\top D_{x_i}G(\xb)\bar{E} \colon i\in \{1,\dots,n\}\right\}=\bar{E}^\top \sym \bar{E}
%	\]
%	holds if, and only if, 
%	\[
%		\Ker(\psi^*)\doteq\left\{M\in S^{m-r}\colon  \langle \bar{E}^\top D_{x_\ell}G(\xb)\bar{E}, M\rangle=0, \ \forall \ell\in \{1,\dots,n\}\right\right\}=\{0\},
%	\]
%	which translates to 
%	\[
%		\sum_{1\leqslant i\leqslant j\leqslant m-r} m_{ij}v_{ij}(\xb)=0 \ \Leftrightarrow \ m_{ij}=0, \forall (i,j)\in \{(i,j)\colon 1\leqslant i\leqslant j\leqslant m-r\},
%	\]
%	given that
%	\[
%		\bar{E}^\top D_{x_\ell}G(\xb)\bar{E}=[(v_{ij}(\xb))_{\ell}]_{i,j\in \{1,\dots,m-r\}},
%	\]
%	where $(v_{ij}(\xb))_{\ell}$ is the $\ell$-th entry of the vector $v_{ij}(\xb)$. However, this is precisely the linear independence of the set $\left\{v_{ij}(\xb)\colon 1\leqslant i\leqslant j\leqslant m-r\right\}$, which is in turn the same as nondegeneracy.
%\end{proof}
%

However, similarly to weak-nondegeneracy, Forsgren's CQ also reduces to LICQ from NLP when $G$ is structurally diagonal (as in~\eqref{sdp:diagg}), contrasting with nondegeneracy. To put Forsgren's CQ in the same terms as the previous sections, we present an elementary characterization of it using the vectors $v_{ij}(\xb,\bar{E})$ defined in Proposition~\ref{prop:ndgvii}:

\begin{proposition}\label{prop:fcqli}
	Let $\xb\in \F$ and let $\bar{E}\in \R^{m\times m-r}$ span $\Ker G(\xb)$. Then, condition \eqref{eq:forsgren1} of Forsgren's CQ 
	%(defined with $\bar{U}\doteq [\bar{E},\bar{U}]$) 
	holds at $\xb$ if, and only if,
	\[
		\sum_{i=1}^{m-r} \sum_{j=i}^{m-r} M_{ij} v_{ij}(\xb,\bar{E})=0, \quad M\in \bar{E}^\T\mathcal{S}(\tilde{G},\xb)\bar{E} \quad \Rightarrow \quad M=0,
	\]
	where $r= \rank(G(\xb))$.
\end{proposition}
\begin{proof}
	Let us assume that $r<m$, since otherwise the proof is trivial. We employ~\cite[Lem. 2]{Forsgren2000}, which states that \jo{\[}\si{$}\bar{E}^\top D_{x_i}G(\xb)\bar{E}=\bar{E}^\top D_{x_i}\tilde{G}(\xb)\bar{E}\si{$}\jo{\]} for all $i\in \{1,\ldots,n\}$, to ensure that the linear operator 
$\psi:\R^n\to \bar{E}^\top \mathcal{S}(\tilde{G},\bar{x}) \bar{E}$,
defined by the action $\psi(d)\doteq \bar{E}^\top DG(\xb)[d]\bar{E}$ is well-defined. %(note that $D_{x_i}\tilde{G}(\xb)\in \mathcal{S}(\tilde{G},\xb)$ for all $i$, so $D\tilde{G}(\xb)[d]\in \mathcal{S}(\tilde{G},\xb)$ for all $d$).

With this in mind, note that 
	\[
		\Im(\psi)=\textnormal{span}\left\{\bar{E}^\top D_{x_i}G(\xb)\bar{E} \colon i\in \{1,\dots,n\}\right\}=\bar{E}^\top \mathcal{S}(\tilde{G},\bar{x}) \bar{E}
	\]
 if, and only if,
	\begin{equation}
		\begin{aligned}
		\Ker(\psi^*) & =\left\{M\in \bar{E}^\T\mathcal{S}(\tilde{G},\xb)\bar{E}\colon  \langle \bar{E}^\top D_{x_\ell}G(\xb)\bar{E}, M\rangle=0, \ \forall \ell\in \{1,\dots,n\}\right\} \jo{ \\ & }=\{0\},
		\end{aligned}
	\end{equation}
	whence the result follows since 
	\[
		\bar{E}^\top D_{x_\ell}G(\xb)\bar{E}=[(v_{ij}(\xb,\bar{E}))_{\ell}]_{i,j\in \{1,\dots,m-r\}},
	\]
	where $(v_{ij}(\xb,\bar{E}))_{\ell}$ is the $\ell$-th entry of the vector $v_{ij}(\xb,\bar{E})$.
		\jo{\hfill\qed}
\end{proof}
As an abuse of language, \eqref{eq:forsgren1} consists of the ``linear independence'' of \si{$}\jo{\[}\left\{v_{ij}(\xb,\bar{E})\colon 1\leqslant i\leqslant j\leqslant m-r\right\}\si{$}\jo{\]} with respect to the set $\bar{E}^\T\mathcal{S}(\tilde{G},\xb)\bar{E}$. In particular, when $G(\xb)=0$ and \eqref{eq:forsgren2} holds, take $\bar{U}=\bar{E}=\I_m$ and note that Forsgren's CQ holds for this particular choice of $\bar{U}$ if, and only if, the set $\left\{\nabla G_{ij}(\xb)\colon (i,j)\in \mathcal{I}(G,\xb)\right\}$ is linearly independent, with $(i,i)\in \mathcal{I}(G,\xb)$ for all $i\in\{1,\dots,m\}$.

\begin{remark}
As far as we understand, the relation between Forsgren's CQ and nondegeneracy was not formally established in~\cite{Forsgren2000}. To clarify this important detail, note that it is clear from Propositions~\ref{prop:fcqli} and~\ref{prop:ndgvii} that nondegeneracy implies Forsgren's CQ. Moreover this implication is clearly strict, as nondegeneracy does not recover LICQ in a diagonal example. 
\end{remark}

The above discussion leads us to deal with sparsity in a more straightforward way, namely without taking Schur complements, which induces another weak variant of nondegeneracy.

\subsection{A sparse variant of nondegeneracy}

For any matrix $\bar{E}$ that spans $\Ker G(\xb)$, consider the function 
\[
	\widehat{G}^{\bar{E}}(x)\doteq \bar{E}^\T G(x)\bar{E}
\]
and note that $\nabla \widehat{G}^{\bar{E}}_{ij}(\xb)=v_{ij}(\xb,\bar{E})$ for all $i,j\in \{1,\ldots,m-r\}$ with $i\leqslant j$. We incorporate structural sparsity into nondegeneracy directly, but in a similar style of Forsgren's CQ (as characterized in~Proposition~\ref{prop:fcqli}), to introduce a new constraint qualification.
\begin{definition}[Sparse-nondegeneracy]\label{def:sndg}
We say that \emph{sparse-nondegeneracy} holds at $\xb\in \F$ when either $\Ker G(\xb)=\{0\}$ or there exists a matrix $\bar{E}\in\R^{m\times m-r}$ that spans $\Ker G(\xb)$ and such that:
\begin{enumerate}
	\item The set $\left\{v_{ij}(\xb,\bar{E})\colon (i,j)\in \mathcal{I}(\widehat{G}^{\bar{E}},\xb), 1\leqslant i\leqslant j\leqslant m-r\right\}$ is linearly independent;
	\item $(i,i)\in \mathcal{I}(\widehat{G}^{\bar{E}},\xb)$ for all $i\in \{1,\ldots,m-r\}$.
\end{enumerate}
\end{definition}

There are two natural questions about sparse-nondegeneracy that we shall answer in the following paragraphs. The first one consists of knowing whether the sparse-nondegeneracy condition is a genuine constraint qualification; and the second one concerns about the relation between Definition~\ref{def:sndg} and other constraint qualifications, such as nondegeneracy, Forsgren's CQ, and Robinson's CQ. To address these questions, we first prove an elementary characterization of sparse-nondegeneracy:
\begin{lemma}\label{lem:sndg}
	Let $\xb\in \F$ be such that $\Ker G(\xb)\neq \{0\}$, and let $\bar{E}$ be a matrix that spans $\Ker G(\xb)$. Then, item 1 of Definition~\ref{def:sndg} holds at $\xb$ if, and only if, there is no nonzero $\tilde{Y}\in \mathcal{S}(\widehat{G}^{\bar{E}},\xb)$ such that $DG(\xb)^*[\bar{E}\tilde{Y}\bar{E}^\T]=0$.
\end{lemma}
\begin{proof}
The result follows directly by noticing that

\begin{equation}
	\begin{aligned}
		%& = \sum_{i,j=1}^{m-r} v_{ij}(\xb,\bar{E})\tilde{Y}_{ij}\\
		\sum_{(i,j)\in\mathcal{I}(\widehat{G}^{\bar{E}},\xb)} v_{ij}(\xb,\bar{E})\tilde{Y}_{ij}
		 & = DG(\xb)^*[\bar{E}\tilde{Y}\bar{E}^\T] 
	\end{aligned}
\end{equation}
for every $\tilde{Y}\in \mathcal{S}(\widehat{G}^{\bar{E}},\xb)$.
	\jo{\hfill\qed}
\end{proof}
Next, we prove that sparse-nondegeneracy implies Robinson's CQ, which also shows that it is indeed a constraint qualification. 
\begin{proposition}\label{prop:sndgrob}
If $\xb\in \F$ satisfies sparse-nondegeneracy, then it also satisfies Robinson's CQ.
\end{proposition}

\begin{proof}
The result follows trivially when $\Ker G(\xb)=\{0\}$, so let us assume that $r= \rank(G(\xb))<m$. Suppose that sparse-nondegeneracy holds at $\xb\in \F$, and take any $Z\succeq 0$ such that $\langle Z,G(\xb)\rangle=0$ and $DG(\xb)^*[Z]=0$, then there exists some $Y\in \S^{m-r}_+$ such that $Z=\bar{E} Y \bar{E}^\T$. Define the matrix $\tilde{Y}\in \mathcal{S}(\widehat{G}^{\bar{E}},\xb)$ whose $(i,j)$-th entry is given by
\[
	\tilde{Y}_{ij}\doteq
	\left\{
	\begin{aligned}
		Y_{ij}, & \  \textnormal{ if } (i,j)\in\mathcal{I}(\widehat{G}^{\bar{E}},\xb)\\
		0, & \ \textnormal{ otherwise,}
	\end{aligned}
	\right.
\]
and note that
\begin{equation}
	\begin{aligned}
		DG(\xb)^*[Z] &= DG(\xb)^*[\bar{E}Y\bar{E}^\T] = DG(\xb)^*[\bar{E}\tilde{Y}\bar{E}^\T]
		 = 0,
	\end{aligned}
\end{equation}
so $\tilde{Y}=0$ due to Lemma~\ref{lem:sndg}. Moreover, from item 2 of Definition~\ref{def:sndg}, $(i,i)\in \mathcal{I}(\widehat{G}^{\bar{E}},\xb)$ for all $i\in\{1,\ldots,m-r\}$, so the diagonal of $Y$ must consist only of zeros, which implies that $Y=0$ and, consequently, $Z=0$. Since $Z$ is arbitrary, Robinson's CQ holds.
	\jo{\hfill\qed} 
\end{proof}

We highlight that item 2 of Definition~\ref{def:sndg} is not superfluous, for removing it may cause us to lose the property of being a constraint qualification. Indeed, the following example illustrates that:

\begin{example}
\label{example:facial}
Consider the problem:
\begin{equation*}
  \begin{aligned}
    & \underset{x \in \mathbb{R}^{2}}{\text{Minimize}}
    & & x_2, \\
    & \text{subject to}
    & & G(x)\doteq \begin{bmatrix}
		x_1 & x_2\\
		x_2 & 0
	\end{bmatrix}\succeq 0,
  \end{aligned}
\end{equation*}
which has $\bar{x}\doteq (0,0)$ as one of its solutions. The point $\xb$ satisfies Definition~\ref{def:sndg} after removing item 2, with $\bar{E}\doteq \I_2$, because $v_{11}(\xb,\bar{E})=(1,0)$ and $v_{12}(\xb,\bar{E})=(0,1)$ are linearly independent; but $\xb$ does not satisfy the KKT conditions since there is no $\bar{Y}\succeq 0$ such that $\bar{Y}_{11}=0$ and $\bar{Y}_{12}=\bar{Y}_{21}=1/2$. Thus, Definition~\ref{def:sndg} is not a constraint qualification without item 2.
\end{example}

\begin{remark}\label{rem:facialreduction}
Let us show that when item 2 fails, the problem can be reformulated such that it holds.
Let $\xb\in \F$ and $\bar{E}$ be a matrix that spans $\Ker G(\xb)$. If item 2 of Definition~\ref{def:sndg} is not satisfied, then let $J\doteq \{i\in \{1,\ldots,m-r\}\colon (i,i)\not\in \mathcal{I}(\widehat{G}^{\bar{E}},\xb)\}$ and note that there exists some $\varepsilon>0$ such that 
\[
	G(x)\in \sym_+ \textnormal{ if, and only if, } G(x)\in \S^m_+\bigcap_{i\in J} \{\bar{e}_i\bar{e}_i^\T\}^\perp,
\]
for every $x\in B(\xb,\varepsilon)$, where $\bar{e}_i$ denotes the $i$-th column of $\bar{E}$. That is, the feasible set $\F$ coincides locally with the preimage of the face $F\doteq\S^m_+\bigcap_{i\in J} \{\bar{e}_i\bar{e}_i^\T\}^\perp$ of $\sym_+$. Moreover, since $F$ is a face of $\S^m_+$, then there is an orthogonal matrix $V\doteq [V_1,V_2]\in \R^{m\times m}$ such that
\[
	V^\T F V=\left\{
	\begin{bmatrix}
		M & 0\\
		0 & 0
	\end{bmatrix}\colon M\in \S^{m-\omega}_+
	\right\},
\]
where $\omega$ is the cardinality of $J$~\cite[Eq. 2.3]{Pataki}. This means that it is possible to locally replace the original constraint of \eqref{NSDP} by the equality constraint $V_2^\top G(x)=0$ and a smaller semidefinite constraint $\mathcal{G}(x)\doteq V_1^\T G(x) V_1\in \S^{m-\omega}_+$. If $F$ is minimal, then the new constraint $\mathcal{G}(x)\in \S^{m-\omega}_+$ satisfies item 2 of Definition~\ref{def:sndg} at $\xb$. Otherwise, this process can be repeated until the minimal face is reached. Thus, every problem can be equivalently reformulated (reducing dimension if necessary), such that item 2 always holds. In particular, when $G$ is an affine function, then this procedure can be computed via a popular preprocessing technique called \emph{facial reduction} (we refer to Pataki~\cite{Pataki} and references therein for more details about it). When $G(\xb)=0$ and $\bar{E}=\I_m$, this procedure can be done by simply removing the $i$-th row and the $i$-th column of $G$, for every $i$ such that $(i,i)\not\in \mathcal{I}(G,\xb)$, and including the correspondent equality constraints into the problem. We recall that all of our results can be easily extended to NSDP problems with separate equality constraints.

Let us illustrate this procedure using Example~\ref{example:facial}. In this case we have $\bar{e}_2^\T G(x)\bar{e}_2=0$ for every $x$; then $x\in \F$ if, and only if, $G(x)\in F$, where
\[
	F\doteq \S^2_+\bigcap 
	\left\{
		\begin{bmatrix}
			0 & 0 \\
			0 & 1
		\end{bmatrix}			
	\right\}^{\perp}=
	\left\{
		\begin{bmatrix}
			\alpha & 0\\
			0 & 0
		\end{bmatrix}
		\colon \alpha \geqslant 0,
	\right\}
\]
which means that the constraint of the problem can be equivalently written as $x_2=0$ and $x_1\geqslant 0$; for which $\xb$ satisfies Definition~\ref{def:sndg} and the KKT conditions.
\end{remark}

\begin{remark}\label{sdp:sndglicq}
If $G$ is structurally diagonal as in~\eqref{sdp:diagg}, then $\xb$ satisfies \jo{\linebreak} sparse-nondegeneracy if, and only if, the set $\{\nabla g_i(\xb)\colon g_i(\xb)=0\}$ is linearly independent. Moreover, this can be extended to block-diagonal constraints. In this case, assuming the same notation as Remark~\ref{rem:wndgblockdiag}, sparse nondegeneracy holds at a feasible point $\xb$ of \eqref{sdp:blockdiag} if, and only if, for each $\ell\in\{1,\ldots,q\}$ there is some matrix $\bar{E}_\ell$ that spans $\Ker G_\ell(\xb)$, such that:

\begin{itemize}
\item For all $i\in \{1,\ldots,m_\ell-r_\ell\}$, we have $(i,i)\in \mathcal{I}(G^{\bar{E}_\ell}_\ell,\xb)$;

\item The set
\[
	\bigcup_{\ell=1}^q \left\{ v_{ij}^{\ell}(\xb,\bar{E}_\ell) \colon (i,j)\in \mathcal{I}(G^{\bar{E}_\ell}_\ell,\xb) \right\}
\]
is linearly independent, where $v_{ij}^\ell(\xb,\bar{E}_\ell)$ is defined as in \eqref{block:viidef}.
\end{itemize}
Note that this is how sparse-nondegeneracy would be defined for a multifold equivalent representation of \eqref{sdp:blockdiag}, with constraints $G_1(x)\succeq 0, \ldots, G_q(x)\succeq 0$.
\end{remark}

In view of Remark~\ref{sdp:sndglicq}, it is easy to build a diagonal counterexample for the converse of Proposition~\ref{prop:sndgrob}. For instance, take $m=2$ and set $\xb=0$; then, define the constraint
\begin{equation}\label{ex:wrobnotwndg}
	G(x)\doteq
	\begin{bmatrix}
		x & 0\\
		0 & x
	\end{bmatrix},
\end{equation}
and note that $v_{11}(\xb,\bar{E})=v_{22}(\xb,\bar{E})=1$ for every matrix $\bar{E}$ that spans $\Ker G(\xb)$. Hence, sparse-nondegeneracy does not hold, although Robinson's CQ does.

Furthermore, Remark~\ref{sdp:sndglicq} reveals a similarity among sparse-nondegeneracy, Forsgren's CQ, and weak-nondegeneracy, which is the fact they all reduce to LICQ when considering a diagonal matrix constraint. Moreover, it follows directly from Propositions~\ref{prop:ndgvii} and~\ref{sdp:ndg} that nondegeneracy also strictly implies sparse-nondegeneracy. To make a rough comparison between Forsgren's CQ and sparse-nondegeneracy, note that both evaluate linear independence of the set $\left\{v_{ij}(\xb,\bar{E})\colon 1\leqslant i\leqslant j\leqslant m-r\right\}$, but while item 1 of Definition~\ref{def:sndg} takes coefficients structured as in $\mathcal{S}(\widehat{G}^{\bar{E}},\xb)$, condition \eqref{eq:forsgren1} takes coefficients structured as in $\bar{E}^\T\mathcal{S}(\tilde{G},\xb)\bar{E}$. This suggests that they are different conditions. In fact, Example~\ref{ex:wndgnotndg} can also be used to show that neither weak- nor sparse-nondegeneracy imply Forsgren's CQ.
\begin{example}[same as Example~\ref{ex:wndgnotndg}]
	Consider the constraint:
	\begin{equation*}
		G(x)\doteq
		\begin{bmatrix}
			x_1 & x_2\\
			x_2 & x_1
		\end{bmatrix}
	\end{equation*}
	and the point $\xb\doteq (0,0)$, which satisfies weak-nondegeneracy and violates nondegeneracy (Example~\ref{ex:wndgnotndg}). Also:
\begin{itemize}
\item \emph{\textbf{Sparse-nondegeneracy holds at $\xb$:}} take the same $\bar{E}$ as above and we have
\[
		\widehat{G}^{\bar{E}}(x)\doteq
		\begin{bmatrix}
			x_1-x_2 & 0\\
			0 & x_1+x_2
		\end{bmatrix}
\]
and $\mathcal{I}(\widehat{G}^{\bar{E}},\xb)=\{(1,1),(2,2)\}$;
	
\item \emph{\textbf{Forsgren's CQ does not hold at $\xb$:}} in this case Forsgren's CQ is equivalent to nondegeneracy, which does not hold because if $\bar{E}\doteq\I_m$, then $v_{11}(\xb,\bar{E})=v_{22}(\xb,\bar{E})=[1,0]$.
\end{itemize}
Thus, neither weak- nor sparse-nondegeneracy imply Forsgren's CQ.
\end{example}

Moreover, if $G(\xb)=0$ then $\tilde{G}=G$ and in this case Forsgren's CQ implies sparse-nondegeneracy (see Proposition \ref{prop:fcqli} and the discussion afterwards). Whether this still holds or not when $G(\xb)\neq 0$ is an open problem that we are currently unable to address, due to the intricate form of $\tilde{G}$ in the general case.

An elementary consequence of Lemma~\ref{lem:sndg} is that sparse-nondegeneracy guarantees uniqueness of the Lagrange multiplier with respect to a fixed sparsity pattern, which is similar to a result proven for Forsgren's CQ~\cite{Forsgren2000}.

\begin{proposition}\label{prop:unique}
	Let $\xb$ be a KKT point of \eqref{NSDP} that satisfies item 1 of Definition~\ref{def:sndg} and let $\bar{E}$ be the matrix that certifies it, which spans $\Ker G(\xb)$. Then, $\Lambda(\xb)\bigcap \left(\bar{E} \mathcal{S}(\widehat{G}^{\bar{E}},\xb)\bar{E}^\T\right)$ is a singleton.
\end{proposition}
\begin{proof}
	Firstly, to see why $\Lambda(\xb)\bigcap \left(\bar{E} \mathcal{S}(\widehat{G}^{\bar{E}},\xb)\bar{E}^\T\right)\neq \emptyset$ we resort to a result of~\cite[Thm. 7]{ahv} which states that under Robinson's CQ any accumulation point $\bar{Y}$ of the sequence 
	\[Y^k \doteq \rho_k \Pi_{\sym_+}(-G(x^k))\]
	must belong to $\Lambda(\xb)$. But clearly, for all $k\in\N$ large enough, we see that $Y^k\in \bar{E}\mathcal{S}(\widehat{G}^{\bar{E}},\xb)\bar{E}^\T$ and so does $\bar{Y}$.
	
	Now let $Y_1,Y_2\in \Lambda(\xb)\bigcap \left(\bar{E} \mathcal{S}(\widehat{G}^{\bar{E}},\xb)\bar{E}^\T\right)$ be Lagrange multipliers associated with $\xb$, define $Y\doteq Y_1-Y_2$, and by definition there exists some $Z\in \mathcal{S}(\widehat{G}^{\bar{E}},\xb)$ such that $Y=\bar{E}Z\bar{E}^\T$ and $DG(\xb)^*[\bar{E}Z\bar{E}^\T]=0$. By Lemma~\ref{lem:sndg} we must have $Z=0$ and, consequently, $Y_1=Y_2$.
		\jo{\hfill\qed}
\end{proof}

Another important property of sparse-nondegeneracy is that the number of structural zeros of $\widehat{G}^{\bar{E}}$, at points that satisfy it, remains the same regardless of $\bar{E}$.
\begin{proposition}\label{prop:sndgsamedim}
Let $\xb\in \F$ be such that $\Ker G(\xb)\neq \{0\}$, and let $\bar{E}$ and $\bar{W}$ be matrices that span $\Ker G(\xb)$, such that item 1 of Definition~\ref{def:sndg} holds. Then, $\#\mathcal{I}(\widehat{G}^{\bar{E}},\xb)=\#\mathcal{I}(\widehat{G}^{\bar{W}},\xb)$.
\end{proposition}
\begin{proof}
Let $Z\doteq [z_{\ell s}]_{\ell,s\in \{1,\ldots,m-r\}}$ be an invertible matrix such that $\bar{E}Z=\bar{W}$ and note that $\widehat{G}^{\bar{W}}(\xb)=Z^\T\widehat{G}^{\bar{E}}(\xb)Z$, so \si{$}\jo{\[}\widehat{G}^{\bar{W}}_{ij}(\xb)=\langle \widehat{G}^{\bar{E}}(\xb),z_i z_j^\T\rangle=\sum_{\ell,s=1}^r z_{\ell i} z_{s j}\widehat{G}^{\bar{E}}_{\ell s}(\xb),\si{$}\jo{\]} where $z_i$ denotes the $i$-th column of $Z$, and
\[
	\nabla \widehat{G}^{\bar{W}}_{ij}(\xb)=\sum_{\ell,s=1}^r z_{\ell i} z_{s j}\nabla \widehat{G}^{\bar{E}}_{\ell s}(\xb).
\]
Rephrasing,
\si{
\[
	\nabla \widehat{G}^{\bar{W}}_{ij}(\xb)=
	\underbrace{\begin{bmatrix}
		\mid & & \mid & \mid & & \mid\\
		\nabla \widehat{G}^{\bar{E}}_{11}(\xb) & \ldots & \nabla \widehat{G}^{\bar{E}}_{m-r,1}(\xb) & \nabla \widehat{G}^{\bar{E}}_{12}(\xb) & \ldots & \nabla \widehat{G}^{\bar{E}}_{m-r,m-r}(\xb)\\
		\mid & & \mid & \mid & & \mid\\
	\end{bmatrix}}_{\doteq \ \textnormal{unfold}(D\widehat{G}^{\bar{E}}(\xb))\colon n\times (m-r)^2}
	\quad \cdot
	\underbrace{\begin{bmatrix}
		 z_{1i} z_{1j}\\
		\vdots\\
		 z_{m-r,i} z_{1j}\\
		 z_{1i} z_{2j}\\
		\vdots\\
		 z_{m-r,i} z_{m-r,j}
	\end{bmatrix}}_{\doteq \ \textnormal{vec}( z_i z_j^\T)\colon (m-r)^2\times 1},
\]
}
\jo{
\[
	\nabla \widehat{G}^{\bar{W}}_{ij}(\xb)=\textnormal{unfold}(D\widehat{G}^{\bar{E}}(\xb)) \textnormal{vec}( z_i z_j^\T)
\]
}
where $\textnormal{unfold}\colon \R^{m-r\times m-r\times n}\to \R^{n\times (m-r)^2}$ is an \textit{unfolding} operator for the tensor $D\widehat{G}^{\bar{E}}(\xb)$ when it is seen as an $m-r\times m-r$ matrix with $n$-dimensional entries. Also, $\textnormal{vec}\colon \R^{m-r\times m-r}\to \R^{(m-r)^2}$ is the usual \textit{vectorization} operator, which transforms a matrix into a vector by stacking up its columns, from left to right. 
\jo{
That is, 
\[
	\textnormal{unfold}(D\widehat{G}^{\bar{E}}(\xb)) \doteq 
	\begin{bmatrix}
		\mid & & \mid & \mid & & \mid\\
		\nabla \widehat{G}^{\bar{E}}_{11}(\xb) & \ldots & \nabla \widehat{G}^{\bar{E}}_{m-r,1}(\xb) & \nabla \widehat{G}^{\bar{E}}_{12}(\xb) & \ldots & \nabla \widehat{G}^{\bar{E}}_{m-r,m-r}(\xb)\\
		\mid & & \mid & \mid & & \mid\\
	\end{bmatrix}
\]
and
\[
	\textnormal{vec}( z_i z_j^\T)\doteq
	\begin{bmatrix}
		 z_{1i} z_{1j}\\
		\vdots\\
		 z_{m-r,i} z_{1j}\\
		 z_{1i} z_{2j}\\
		\vdots\\
		 z_{m-r,i} z_{m-r,j}.
	\end{bmatrix}
\]
}
Consequently,
\si{
\[
	\underbrace{\textnormal{unfold}(D\widehat{G}^{\bar{W}}(\xb))}_{n\times (m-r)^2}=\textnormal{unfold}(D\widehat{G}^{\bar{E}}(\xb))\cdot 		
	\underbrace{\begin{bmatrix}
		\mid & & \mid & \mid & & \mid\\
		\textnormal{vec}( z_1 z_1^\T) & \ldots& \textnormal{vec}( z_1 z_r^\T) & \textnormal{vec}( z_2 z_1^\T) & \ldots& \textnormal{vec}( z_r z_r^\T)\\
		\mid & & \mid & \mid & & \mid\\
	\end{bmatrix}}_{(m-r)^2\times(m-r)^2},
\]
which can be rephrased in terms of the \textit{Kronecker product} as $\textnormal{unfold}(D\widehat{G}^{\bar{W}}(\xb))=\textnormal{unfold}(D\widehat{G}^{\bar{E}}(\xb)) Z\otimes Z$. 
}
\jo{
\[
	\textnormal{unfold}(D\widehat{G}^{\bar{W}}(\xb))=\textnormal{unfold}(D\widehat{G}^{\bar{E}}(\xb)) (Z\otimes Z),
\]
where
\[	
	Z\otimes Z\doteq \begin{bmatrix}
		\mid & & \mid & \mid & & \mid\\
		\textnormal{vec}( z_1 z_1^\T) & \ldots& \textnormal{vec}( z_1 z_r^\T) & \textnormal{vec}( z_2 z_1^\T) & \ldots& \textnormal{vec}( z_r z_r^\T)\\
		\mid & & \mid & \mid & & \mid\\
	\end{bmatrix},
\]
is the \textit{Kronecker product} of $Z$ with itself.
}
But since $Z$ is invertible, $Z\otimes Z$ is also invertible, which means that
\begin{equation*}
	\begin{aligned}
	\textnormal{span}\left(\left\{\nabla \widehat{G}^{\bar{W}}_{ij}(\xb)\colon 
	\si{1\leqslant} i\leqslant j\si{\leqslant m-r}
	\right\}
	\right) & = \textnormal{span}\left(\left\{\nabla \widehat{G}^{\bar{E}}_{ij}(\xb)\colon 
	\si{1\leqslant} i\leqslant j\si{\leqslant m-r}
	\right\}\right).\\
	\end{aligned}
\end{equation*}
Then, since $\nabla \widehat{G}^{\bar{E}}_{ij}(\xb)=0$ for all $(i,j)\in \mathcal{I}(\widehat{G}^{\bar{E}},\xb)$ (and the same holds for $\bar{W}$), it follows that 
\begin{equation*}
	\begin{aligned}
	\textnormal{span}\left(\left\{\nabla \widehat{G}^{\bar{W}}_{ij}(\xb)\colon (i,j)\in \mathcal{I}(\widehat{G}^{\bar{W}},\xb)\right\}\right) & = \textnormal{span}\left(\left\{\nabla \widehat{G}^{\bar{E}}_{ij}(\xb)\colon (i,j)\in \mathcal{I}(\widehat{G}^{\bar{E}},\xb)\right\}\right)\si{,}.
	\end{aligned}
\end{equation*}
Finally, since item 1 of Definition~\ref{def:sndg} holds for both $\bar{E}$ and $\bar{W}$, we conclude that $\#\mathcal{I}(\widehat{G}^{\bar{E}},\xb)=\#\mathcal{I}(\widehat{G}^{\bar{W}},\xb)$.
	\jo{\hfill\qed}
\end{proof}

Proposition~\ref{prop:sndgsamedim} tells us that the strength of sparse-nondegeneracy is invariant with respect to $\bar{E}$. That is, if there are multiple matrices $\bar{E}$ certifying sparse-nondegeneracy at a point $\xb$, then they all induce similar conditions. In our opinion, this is an advantage with respect to Forsgreen's CQ. As for weak-nondegeneracy, we were not able to find any counterexample nor prove any relation between them. In fact, finding this relation seems a challenging task since there is no clear relation between the eigenvectors of $G(x)$ and its sparsity structure, in general.

One should also keep in mind that if sparse-nondegeneracy holds at some $\xb$, then it also holds in a neighborhood of $\xb$.

\begin{theorem}\label{thm:stabsparse}
Let $\xb\in \F$ satisfy sparse-nondegeneracy. Then, there exists a neighborhood $\mathcal{V}$ of $\xb$ such that every $x\in\mathcal{V}$ satisfies sparse-nondegeneracy. 
\end{theorem}

\begin{proof}
Suppose that the statement above is false. That is, suppose that there exists a feasible sequence $\seq{x}\to \xb$ such that sparse-nondegeneracy fails at each $x^k$, but it holds at $\xb$. Our aim is to prove that this leads to an absurd. So let $\bar{E}$ be any matrix with orthonormal columns that span $\Ker G(\xb)$ and, for each $k\in \N$ let $\Pi^k$ be the projection matrix onto the space spanned by the $m-r$ smallest eigenvectors of $G(x^k)$, which is well defined when $k$ is sufficiently large. Define $\tilde{W}^k\doteq \Pi^k \bar{E}$, for all such $k\in \N$. It is well-known (see, for instance,~\cite[Ex. 3.98]{bshapiro}) that the columns of $\tilde{W}^k$ are linearly independent, which allows us to apply the Gram-Schmidt orthonormalization process to them and arrange its output in the columns of a new matrix, which we will denote by $W^k$. It is also known that $W^k\to \bar{E}$ as $k\to \infty$.

Because sparse-nondegeneracy fails at $x^k$, we know that the rank $r^k$ of $G(x^k)$ is smaller than $m$, and by the pigeonhole principle we can even assume that $r^k$ is the same, say $\tilde{r}$, for every $k\in \N$. Also, note that $m-\tilde{r}\leqslant m-r$ and that, by construction, we can assume that the first $m-\tilde{r}$ columns of each $W^k$, which we will arrange in a matrix denoted by $E^k$, span $\Ker G(x^k)$. Since sparse-nondegeneracy fails at $x^k$ it holds that $\{v_{ij}(x^k,E^k)\}_{(i,j)\in \mathcal{I}(\widehat{G}^{E^k},x^k)}$ linearly dependent for each $k\in \N$. Observe that since $\lim_{k\to\infty} E^k$ is a submatrix of $\bar{E}$ we have that 
\[
	\lim_{k\to \infty}\mathcal{I}(\widehat{G}^{E^k},x^k)\subseteq \mathcal{I}(\widehat{G}^{\bar{E}},\xb)
\]
and
\[
	\lim_{k\to \infty} \{v_{ij}(x^k,E^k)\}_{(i,j)\in \mathcal{I}(\widehat{G}^{E^k},x^k)}\subseteq \{v_{ij}(\xb,\bar{E})\}_{(i,j)\in \mathcal{I}(\widehat{G}^{\bar{E}},\xb)}.
\]
The left-hand side of the expression above is linearly dependent, which makes $\{v_{ij}(\xb,\bar{E})\}_{(i,j)\in \mathcal{I}(\widehat{G}^{\bar{E}},\xb)}$ linearly dependent as well. Because $\bar{E}$ is arbitrary, it follows that sparse-nondegeneracy fails at $\xb$, which is a contradiction.
	\jo{\hfill\qed}
\end{proof}

\begin{remark}
It is noteworthy that it is also possible to define another variant of Robinson's CQ that enjoys sparsity, by replacing $\tilde{Y}\in \mathcal{S}(\widehat{G}^{\bar{E}},\xb)$ by $\tilde{Y}\in \mathcal{S}(\widehat{G}^{\bar{E}},\xb)\cap \S^{m-r}_+$ in Lemma~\ref{lem:sndg}. This definition is strictly implied by sparse-nondegeneracy (see the example given in \eqref{ex:wrobnotwndg}). Moreover, it is clear that this variant of Robinson's CQ is implied by Robinson's CQ, but the converse is also an open question. The proof that this is a CQ follows similarly to the proof of Theorem~\ref{thm:wndgcq}.
\end{remark}

\begin{remark}
Regarding second-order optimality conditions, we call the reader's attention to the fact that for each $\xb\in\ F$ and each $\bar{E}$ that spans $\Ker G(\xb)$, there exists a neighborhood $\mathcal{V}$ of $\xb$ such that 
\[
	G(\mathcal{V})=G(\mathcal{V})\cap \underbrace{\bar{E}\mathcal{S}(\widehat{G}^{\bar{E}},\xb)\bar{E}^\T}_{\doteq H}
\]
This means that near $\xb$ we can consider a new space $\S^H\doteq \sym\cap H$, define a new cone $\S^H_+\doteq \sym_+\cap H$ which is still closed and convex, and in this setting sparse-nondegeneracy induces a second-order optimality condition, which is inherited from~\cite[Thm. 3.45]{bshapiro}. Namely, for every $d\in DG(\xb)^{-1}(T_{\sym_+\cap H}(G(\xb))\cap \{\nabla f(\xb)\}^\perp$ it holds that
\begin{equation}\label{eq:basicsoc}
\sup_{Y\in \Lambda(\xb)\cap H} \left(d^\T\nabla^2 L(\xb,Y)d - \sigma(Y,T^{2}_{\sym_+\cap H}(G(\xb),DG(\xb)[d]))\right)\geqslant 0,
\end{equation}
because sparse-nondegeneracy implies Robinson's CQ, which in turn is carried over to the reduced problem, but since $\Lambda(\xb)\cap H$ is a singleton, we have for $\bar{Y}\in \Lambda(\xb)\cap H$ that 
\[	
d^\T\nabla^2 L(\xb,\bar{Y})d - \sigma(\bar{Y},T^{2}_{\sym_+\cap H}(G(\xb),DG(\xb)[d]))\geqslant 0.
\]
Although this condition concerns the reduced problem, mostly, it can also bring some information about the original problem, for an inequality analogous to~\eqref{eq:basicsoc} in terms of $\sup_{Y\in \Lambda(\xb)}$ is also true. Above, \si{\linebreak}$T^{2}_{\sym_+\cap H}(G(\xb),DG(\xb)[d])$ denotes the \textit{second-order tangent set} to $\sym_+\cap H$ at $G(\xb)$ along $DG(\xb)[d]$ (see~\cite[Def. 3.28]{bshapiro}), and $\sigma(Y,T^{2}_{\sym_+\cap H}(G(\xb),DG(\xb)[d]))$ denotes its support function.
\end{remark}

\subsection{Zeros of the gradients and sparse-nondegeneracy}

In this short ending section, we discuss how to improve sparse-nondegeneracy even further. This is mainly motivated by the realization that the idea of disregarding ``structural zeros'' in the study of regularity is actually too conservative. Since nondegeneracy is mainly concerned with the derivative of $G$ at $\xb$ instead of the value of $G$ in a neighborhood of $\xb$, we can in fact ignore all entries of $G$ whose gradients are zero at $\xb$, which is done by considering the following sets:
\[
	\begin{array}{ll}
		\mathcal{S}_\nabla(F,\xb)&\doteq\left\{M\in \S^\beta\colon M_{ij}=0 \text{ if } \nabla F_{ij}(\xb)=0\right\}
	\end{array}
\]
and
\[
	\mathcal{I}_\nabla (F,\xb)\doteq
	\left\{(i,j) \colon \nabla F_{ij}(\xb)\neq 0, \ 1\leqslant i\leqslant j\leqslant \beta\right\}.
\]
For example, if $n=1$ and $\beta=3$, for all $x$ close to $\xb\doteq 0$ we have, as an example,
\begin{equation}\label{sparse:forsgren}
	\textnormal{if } F(x)\doteq
	\begin{bmatrix}
		x & 0 & x^2\\
		0 & x & 1\\
		x^2 & 1 & x
	\end{bmatrix}
	\textnormal{ then } 
	M\in \mathcal{S}(F,\xb)
	\Leftrightarrow
	M=
	\begin{bmatrix}
		M_{11} & 0 & 0\\
		0 & M_{22} & 0\\
		0 & 0 & M_{33}
	\end{bmatrix},
\end{equation}
where $M_{11}, M_{22},$ and $M_{33}$ may or may not be zero. Then, we can define a condition similarly to Definition~\ref{def:sndg} but in terms of $\mathcal{I}_\nabla$:

\begin{definition}[GS-nondegeneracy]\label{def:ewsndg}
We say that the condition \emph{gradient-sparse-nondegeneracy} \si{\linebreak} (GS-nondegeneracy) holds at $\xb\in \F$ if either $\Ker G(\xb)=\{0\}$ or there exists a matrix $\bar{E}\in\R^{m\times m-r}$ that spans $\Ker G(\xb)$ such that:
\begin{enumerate}
	\item The set $\left\{v_{ij}(\xb,\bar{E})\colon (i,j)\in \mathcal{I}_\nabla(\widehat{G}^{\bar{E}},\xb), 1\leqslant i\leqslant j\leqslant m-r\right\}$ is linearly independent;
	\item $(i,i)\in \mathcal{I}_\nabla(\widehat{G}^{\bar{E}},\xb)$ for all $i\in \{1,\ldots,m-r\}$.
\end{enumerate}
\end{definition}

 The interesting properties of GS-nondegeneracy that make it worth an extended comment are twofold. The first one is that sparse-nondegeneracy is strictly stronger than GS-nondegeneracy. Noticing that $\mathcal{I}_\nabla(\widehat{G}^{\bar{E}},\xb)\subseteq \mathcal{I}(\widehat{G}^{\bar{E}},\xb)$ is enough to see the implication and the next example shows that the converse is not necessarily true.

\begin{example}\label{ex:nino}
Let 
\[
	G(x)\doteq \begin{bmatrix}
		x_1 & x_2^2\\
		x_2^2 & x_2
	\end{bmatrix}
\]
and consider the constraint $G(x)\succeq 0$ at the point $\xb\doteq (0,0)$. In this case, Forsgren's CQ fails at $\xb$ with $\bar{E}\doteq \I_2$ because
\[
\textnormal{span}\left\{ \begin{bmatrix}
1 & 0\\ 0 & 0
\end{bmatrix},\begin{bmatrix}
0 & 0\\ 0 & 1
\end{bmatrix} \right\}\neq \mathcal{S}(\tilde{G},\xb)=\mathbb{S}^2.
\]
In fact, \eqref{eq:forsgren1} fails for every orthogonal matrix $\bar{E}$. Furthermore, regardless of $\bar{E}$ the vectors $v_{11}(\xb,\bar{E}),v_{22}(\xb,\bar{E})$, and $v_{12}(\xb,\bar{E})\in \R^2$, are linearly dependent and $\mathcal{I}(\widehat{G}^{\bar{E}},\xb)=\{(1,1),(1,2),(2,2)\}$, hence sparse-nondegeneracy also fails to hold at $\xb$. On the other hand, note that for $\bar{E}=\I_2$, we obtain $\mathcal{I}_{\nabla}(\widehat{G}^{\bar{E}},\xb)=\{(1,1),(2,2)\}$ and 
\[
	v_{11}(\xb,\bar{E})=
	\begin{bmatrix}
		1 \\
		0	
	\end{bmatrix} \quad \textnormal{and} \quad v_{22}(\xb,\bar{E}) =\begin{bmatrix}
		0 \\
		1	
	\end{bmatrix} 
\] are linearly independent, so GS-nondegeneracy holds at $\xb$.
\end{example}

We remark that Lemma \ref{lem:sndg} and Propositions~\ref{prop:sndgrob}, \ref{prop:unique}, and  \ref{prop:sndgsamedim}, can be also stated and proved in terms of GS-nondegeneracy. Moreover, if Forsgren's CQ was defined in terms of $\mathcal{I}_\nabla(\tilde{G},\xb)$ instead of $\mathcal{I}(\tilde{G},\xb)$, we would obtain precisely Definition \ref{def:ewsndg} (due to~\cite[Lem. 2]{Forsgren2000}), which is quite unexpected. The second interesting aspect of GS-nondegeneracy is that, although an analogue of Theorem~\ref{thm:stabsparse} may not be true, it presents at least a different notion of stability, in the sense of ignoring small perturbations. Formally:

\begin{theorem}\label{thm:stabews}
Let $\xb\in \F$ and $\delta\colon \R^n\to \mathbb{S}^m$ be any continuously differentiable function such that $\delta(\xb)=0$ and $D \delta(\xb)=0$. Then, GS-nondegeneracy holds at $\xb$ for the constraint $G(x)\succeq 0$ if, and only if, it holds for the constraint $G_{\delta}(x)\doteq G(x)+\delta(x)\succeq 0$ at the same point.
\end{theorem}

\begin{proof}
Direct from the fact $DG(\xb) = DG_\delta(\xb)$ and $\mathcal{I}_{\nabla} (\widehat{G}^{\bar{E}},\xb)=\mathcal{I}_{\nabla} (\widehat{G}_\delta^{\bar{E}},\xb)$.
\end{proof}

Despite the apparent triviality of Theorem \ref{thm:stabews}, observe that it is essentially telling us that any noise of order two can be disregarded, as we could observe in Example~\ref{ex:nino}.

\section{Conclusions}\label{sec:conc}

In this paper, we studied the nondegeneracy condition of Shapiro and Fan~\cite{shapfan} with the purpose of incorporating some matrix structure into it, such as spectral decompositions and structural sparsity. Our work was motivated by a well-known limitation of nondegeneracy, which is the fact it generally fails in the presence of structural sparsity in the constraint function. For example, we recall that a NSDP problem with multiple constraints may be equivalently reformulated as a single block diagonal constraint, but nondegeneracy is not expected to be preserved in the process. This limitation may have important consequences in practice, since many algorithms are theoretically supported by nondegeneracy and, on the other hand, structural sparsity is a very common trait of optimization models of real world problems.

%Shapiro~\cite[Prop. 6]{Shapiro1997} provided a characterization of nondegeneracy in terms of the gradients of the entries of the matrix function $\bar{E}^\T G(x)\bar{E}$, where $\bar{E}$ is any matrix that spans $\Ker G(\xb)$. We show that it is enough to consider only the gradients of the entries of $\bar{E}^\T G(x)\bar{E}$ that are not structural zeros, for some $\bar{E}$, as long as they are not in its diagonal. Moreover, we also show that if we take an appropriate $\bar{E}$ for each sequence $\seq{x}$ that converges to $\xb$, then we may only evaluate the gradients of the diagonal entries of $\bar{E}^\T G(x)\bar{E}$. 

To address this issue, we proposed three variants of nondegeneracy, here called weak-nondegeneracy, sparse-nondegeneracy, and GS-nondegeneracy. They were proven to be strictly weaker than the classical nondegeneracy. In particular, all new constraint qualifications only require the dimension constraint $n\geqslant m-r$, which is considerably less demanding than the constraint $n\geqslant (m-r)(m-r+1)/2$ imposed by nondegeneracy. Also, they are invariant to multifold or block diagonal formulations of \eqref{NSDP} and, consequently, they recover the LICQ condition from NLP when the constraint function is structurally diagonal.

All our conditions are inspired by {\it sequential optimality conditions} \cite{ahm10,ahv} which provide simple proofs for the facts that the conditions we define are CQs (the proof for sparse-nondegeneracy and GS-nondegeneracy were not presented but they are left for the reader). Besides the simplicity of the approach, the convergence of an external penalty method to KKT points under these CQs is obtained automatically (see the discussion after Theorem~\ref{thm:wndgcq}), which is a direct application of the new CQs. Also, several other CQs for NLP have been recently (re)invented with sequential optimality conditions in mind. In particular, the so-called \textit{constant rank constraint qualification} (CRCQ) by Janin~\cite{janin}, and the \textit{constant positive linear dependence} (CPLD) of Qi and Wei~\cite{qiwei}, together with their weaker counterparts \cite{cpg,rcpld,rcrcq}. Previous attempts have been made to extend these CQs to the conic context, but they have turned out to be flawed \cite{errata} or incomplete \cite{crcq-naive}, since the results in \cite{crcq-naive} are only relevant for multifold conic problems where at least one block of constraints is such that the zero eigenvalue is simple. The approach we present in this paper gives the proper tools for providing the extension of all mentioned CQs to the context of general NSDPs and, more generally, to optimization over symmetric cones, also extending the global convergence results to more practical algorithms. For instance, in NLP, it is known that the convergence theory of a safeguarded augmented Lagrangian method can be built around CPLD~\cite{abms}, which will also be the case for its NSDP variant \cite{ahv}. A continuation of this paper will appear shortly with these results.

With this in mind, we believe that the concepts introduced in this paper are interesting enough to shed a new light to the classical theme of constraint nondegeneracy for conic programming, showing, in particular, how to redefine it in such a way that linear independence can be replaced by weaker notions. In this process, new and interesting challenging open questions have appeared which we believe should be addressed. In particular, new studies should be conducted to clarify the relationship between weak-nondegeneracy and sparse-nondegeneracy, together with the relationship between weak-Robinson's CQ and Robinson's CQ (see Figure~\ref{fig:relationsCQ}).

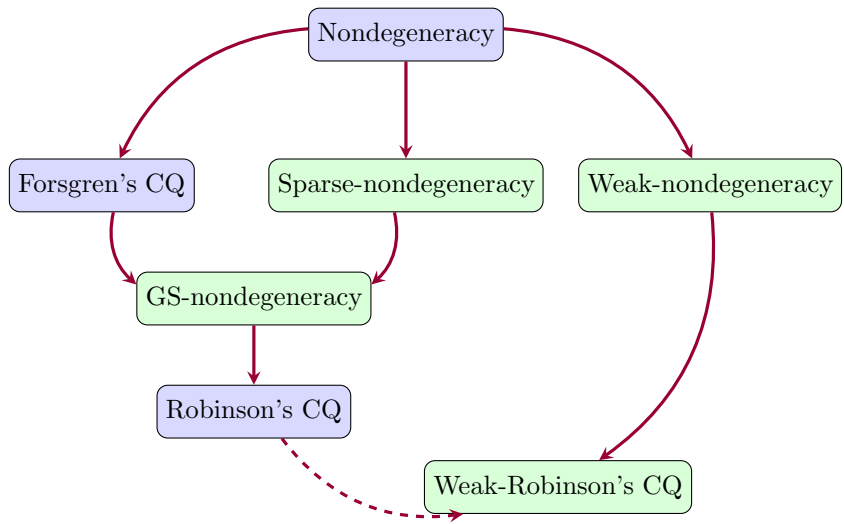
\begin{figure}[!htb]
	\centering
	\begin{tikzpicture}
		\tikzset{% This is the style settings for nodes
		    old/.style={rectangle,rounded corners, minimum size=0.7cm,fill=blue!15,draw=black},
		    new/.style={rectangle,rounded corners, minimum size=0.7cm,fill=green!15,draw=black},
		    v/.style={-stealth,very thick,purple!80!black},
		    u/.style={-stealth,dashed,very thick,purple!80!black},
		}
		\node[old] (NDG) at (0,0) {Nondegeneracy};
		\node[old] (RCQ) at (-2,-5) {Robinson's CQ};
		\node[new] (EWS) at (-2,-3.5) {GS-nondegeneracy};
		\node[old] (FCQ) at (-4,-2) {Forsgren's CQ};
		\node[new] (WNDG) at (4,-2) {Weak-nondegeneracy};
		\node[new] (SNDG) at (0,-2) {Sparse-nondegeneracy};
		\node[new] (WRCQ) at (2,-6) {Weak-Robinson's CQ};
		%\node[new] (CRCQ) at (0,-2.6) {Conic CRCQ};
		%\node[new] (CPLD) at (0,-5.2) {Conic CPLD};
		%
		\draw[v] (NDG) to[bend left] (WNDG);
		\draw[v] (NDG) to[bend right] (FCQ);
		\draw[v] (NDG) -- (SNDG);	
		\draw[v] (SNDG) to[bend left] (EWS);
		%\draw[v] (WNDG) -- (CRCQ);
		%\draw[v] (CRCQ) -- (CPLD);
		%
		\draw[v] (FCQ) to[bend right] (EWS);
		\draw[u] (RCQ) to[bend right] (WRCQ);
		\draw[v] (EWS) -- (RCQ);
		%
		%\draw[v] (NDG) to[bend right] (RCQ);
		%\draw[v] (WNDG) to[bend left] (WRCQ);
		\draw[v] (WNDG) to[bend left] (WRCQ);
	\end{tikzpicture}
	\captionsetup{justification=raggedright,margin=3cm}
	\caption{Relationship among some CQs for NSDP. Classical CQs are in blue boxes, while new CQs are in green boxes. Arrows indicate strict implications, except for the dashed arrow where the reverse implication is unknown.}
	\label{fig:relationsCQ}
\end{figure}

\bibliographystyle{plain}

\end{document}